\documentclass{GT}
\usepackage{amssymb}
\usepackage{graphics}
\usepackage{epstopdf}
\usepackage{graphicx}
\newtheorem{theorem}{Theorem}[section] 
\newtheorem{lemma}[theorem]{Lemma}     
\newtheorem{corolary}[theorem]{Corollary}
\newtheorem{proposition}[theorem]{Proposition}

\newnumbered{assertion}{Assertion}    
\newnumbered{conjecture}{Conjecture}  
\newnumbered{definition}{Definition}
\newnumbered{hypothesis}{Hypothesis}
\newnumbered{remark}{Remark}
\newnumbered{note}{Note}
\newnumbered{observation}{Observation}
\newnumbered{problem}{Problem}
\newnumbered{question}{Question}
\newnumbered{algorithm}{Algorithm}
\newnumbered{example}{Example}
\newunnumbered{notation}{Notation}
\newcommand{\R}{{\mathbb R}}

\newcommand{\N}{{\mathbb N}}
\newcommand{\C}{{\mathbb C}}
\newcommand{\K}{{\mathbb K}}
\def\codim{\mathop{\mbox{\normalfont codim}}\nolimits}
\def\Int{\mathop{\mbox{\normalfont Int}}\nolimits}

\title{Global continuation of monotone wavefronts} 

\author{Adrian Gomez and Sergei Trofimchuk}

\classno{34K12, 35K57, 92D25}

\begin{document}
\maketitle

\begin{abstract}
\noindent In this paper, we answer the question about
the criteria of existence of monotone travelling fronts  $u = \phi(\nu \cdot
x+ct), \ \phi(-\infty) =0, \phi(+\infty) = \kappa,$ for the
monostable (and, in general, non-quasi-monotone) delayed
reaction-diffusion equations $u_t(t,x) - \Delta u(t,x)  =
f(u(t,x), u(t-h,x)).$ $C^{1,\gamma}$-smooth $f$ is supposed to satisfy $f(0,0) =
f(\kappa,\kappa) =0$ together with other monostability
restrictions.  Our theory covers the two most important cases:
Mackey-Glass type  diffusive equations and  KPP-Fisher  type equations.
The proofs are  based on a variant of Hale-Lin
functional-analytic approach to the heteroclinic solutions  where
Lyapunov-Schmidt reduction  is realized in a `mobile' weighted
space of $C^2$-smooth functions.  This  method requires a detailed
analysis of  a family of associated linear differential Fredholm
operators: at this stage, the discrete Lyapunov
functionals by Mallet-Paret and Sell are used  in an essential way. \end{abstract}


\section{Introduction and main result}
\label{intro}

The aim of this paper is  to obtain efficient
criteria of existence of monotone travelling waves  $u =
\phi(\nu \cdot x+ct), \ \phi(-\infty) =0,$ $ \phi(+\infty) = \kappa
>0,$ for the non-quasi-monotone  functional  reaction-diffusion
equations
\begin{equation}\label{pe}
u_t(t,x) = \Delta u(t,x)  + f(u(t,x), u(t-h,x)), \ u \geq 0,\ x
\in \R^m,
\end{equation}
 in that case when the function $g(x):= f(x,x)$ is of
non-degenerate monostable type:  $g(0) = g(\kappa) =0, \ g'(0)>0,
g'(\kappa) <0$,  and $g(x) >0$ for $x \in (0, \kappa)$.  Here $\nu
\in \R^m$ is a fixed unit vector, $c>0$ is the propagation
speed and $h\geq 0$ is the delay. Henceforth we will assume that
$f$ is $C^{1,\gamma}$-smooth  function, $\gamma \in(0,1]$.

There is a long list of studies that consider the wavefront
existence  for equation (\ref{pe})  either with or without delays,
let us mention here only several of them:
\cite{ZAMP,fhw,GK,GT,HT,KO,LZii,ma1,TPT,TTT,wz,z}.
The problem  is quite well understood   when  $h=0$. In
particular, there exists $c^{\frak N}_*>0$ (called the minimal
speed of propagation) such that, for every $c \geq c^{\frak N}_*$,
equation (\ref{pe})  has exactly one wavefront  $u = \phi(\nu
\cdot x+ct)$, see \cite[Theorems 8.3(ii) and 8.7]{GK} or
\cite{LZii,TPT}. In addition, (\ref{pe}) does not have any  front
propagating at the velocity $c < c^{\frak N}_*$. There are several
variational principles describing  $c^{\frak N}_*$
\cite{BD,GK}. If $g(x) \leq g'(0)x,$ $ x \geq 0,$ then
$c^{\frak N}_*= 2\sqrt{g'(0)}$. In general, however, simple
analytical formulas for   $c^{\frak N}_*$  are not available.
 The
profile $\phi$  is necessarily strictly increasing  \cite[Theorem
2.39]{GK} and the following asymptotic formulae are valid
\cite{TPT} for $c > c^{\frak N}_*$ and appropriate $s_j =s_j(c,
\phi),\ \sigma >0:$
\begin{eqnarray}\label {afe}
(\phi, \phi')(t+s_0,c)&=& e^{\lambda(c) t}(1, \lambda(c)) + O(e^{(\lambda(c)+ \sigma) t}), \ t \to -\infty, \nonumber \\
(\phi, \phi')(t+s_1,c)&=& (\kappa,0) - e^{\lambda_2(c) t}(1,
\lambda_2(c)) + O(e^{(\lambda_2(c)- \sigma) t}), \ t \to +\infty.
\end{eqnarray}
Here $\lambda(c)$ [respectively,  $\lambda_2(c)$] is the closest
to $0$  positive [respectively, negative] zero of the
characteristic polynomial $z^2-cz +g'(0)$ [respectively, $z^2-cz
+g'(\kappa)$].

However,  when $h>0$, there are numerous gaps in our knowledge
about the wavefronts of equation (\ref{pe}). As for now, neither
of  the questions concerning the existence, uniqueness, geometric
shape of fronts has been completely answered even for such quite
studied models as the Nicholson's blowflies diffusive equation
\cite{AGT,ma1,ML,TTT,z} and the KPP-Fisher delayed equation
\cite{ZAMP,BNPR,CMP,FZ,fhw,GT,KO}. An additional
complication appearing in the delayed case is the
possible non-monotonicity of wavefronts
\cite{ZAMP,BNPR,TTT}. But even  the existence of {\it  monotone}
fronts is usually  proved only under the quasi-monotonicity
assumption on $f(u,v)$.  In particular,  it is an open problem whether
the minimal  speed of  propagation $c^{\frak N}_*>0$ for (\ref{pe})  can be
well defined in the situation when  $f(u,v)$ is not quasi-monotone  and is not
dominated by its linear part at $(0,0)$ (cf.  \cite{TPT} and Lemma \ref{tdo} below).  In fact,
even in the case of  quasi-monotone nonlinearities,  $c^{\frak N}_*>0$  was defined in full generality
only very recently, in the fundamental contribution \cite{LZii} by X. Liang  and X.-Q. Zhao.
Another example:  due to the relatively
`bad' monotonicity properties of $f(u,v) = u(1-v)$, an efficient criterion of
existence of monotone wavefronts to the delayed KPP-Fisher
equation was obtained just   a few years ago  \cite{FZ,GT,KO} (in Section \ref{apa},
we present a significant  extension of this  result).
For the Nicholson's blowflies equation where  $f(u,v) = -u +pve^{-v}, p >e,$ the
similar question was not yet answered: in Section \ref{apa}, we present a
complete solution to the existence problem when $p \in (e,e^2]$ and we describe partially
this solution when $p > e^2$.

Now, there are very few approaches which can be used to address
the wavefront existence  for equation (\ref{pe}).  It should be noted that 
the profile
$\phi$ of   travelling front $u(t,x) = \phi(\nu \cdot
x+ct),$ $ \phi(-\infty) =0,$ $ \phi(+\infty) = \kappa >0,$ defines
a heteroclinic solution of the delay differential equation
\begin{equation}\label{e1}
\phi''(t)-c\phi'(t) + f(\phi(t),\phi(t-ch))=0, \ t \in \R.
\end{equation} 
Therefore the 
phase plane analysis, which is usually invoked in the non-delayed case, does not work when $h>0$ because of the infinite
dimension of phase spaces associated  to 
equation (\ref{e1}).  As a consequence, several alternative ideas were proposed, see e.g. \cite{BNPR,fhw,KO,wz}. Between them,  the upper-lower solution method \cite{CMP,GT,ma1,wz}  and a perturbation approach based on the Lyapunov-Schmidt procedure \cite{fhw,FTnon,HL} are  the most used by the researchers.
The latter method relies  essentially on the fact that  delay
differential equation (\ref{e1}) simplifies in the limit cases $c
= +\infty$ and $h=0$. For instance,  the limit form (as $c \to
+\infty$) of (\ref{e1}) is $\phi'(t) = f(\phi(t),\phi(t-h))$.
Assume that this equation  linearized along  its heteroclinic
solution $\psi$ defines a surjective Fredholm operator  in an
appropriate Banach space. In consequence, the Lyapunov-Schmidt
reduction can be used to prove the existence of a smooth family of
{fast} (i.e. $c>c_*$ for some large $c_*$) wave solutions in some
neighborhood of $\psi$. We remark that the value of $c_*>0$  is
at least very difficult to compute or estimate.
Therefore, the existence results obtained by this technique so far
have local nature (e.g., the existence is proved for velocities in
some neighborhood of $c=+\infty$).  This constitutes a serious
drawback for the applications  because of the special importance
that the minimal fronts have for the description
of propagation phenomena.   Nevertheless, as we show in this
paper, the described approach still can be extended  to prove the
existence of the global families of wavefronts for several important classes
of equations.  The key property of wavefronts which is needed
for the mentioned  extension  is their monotonicity.  It seems that our
methodology does not apply to  non-monotone travelling
fronts.

Before stating the main theorem of this work,  we  need to discuss several
properties of  the spectra of   the following linearizations of
(\ref{e1}) along the equilibria $0, \kappa:$
\begin{equation}\label{ve1}
v''(t)-cv'(t) + \alpha_j v(t) + \beta_j v(t-ch)=0, \quad j \in
\{0, \kappa\}.
\end{equation}
Here $ \alpha_0:= f_1(0,0), \ \beta_0:= f_2(0,0), \
\alpha_\kappa: = f_1(\kappa,\kappa), \  \beta_\kappa:=
f_2(\kappa,\kappa)$ and $f_j(x_1,x_2):= f_{x_j}(x_1,x_2)$. Recall
that  the monostable function $g(x) :=f(x,x)$ satisfies
\[
g'(0) = f_1(0,0)+ f_2(0,0) = \alpha_0+\beta_0 >0, \ g'(\kappa) =
f_1(\kappa,\kappa)+ f_2(\kappa,\kappa)= \alpha_\kappa+\beta_\kappa
< 0.
\]
Additionally, in view of applications in population dynamics (see
Section \ref{apa}), we will assume that $ \beta_0 = f_2(0,0) \geq
0. $
\begin{lemma} \label{lc2} Given $\alpha_\kappa+ \beta_\kappa <0, \ \beta_\kappa < 0,$ there exists
$c^{\frak L} _\kappa=c_\kappa^{\frak L} (h) \in (0, +\infty]$ such
that the characteristic equation
\begin{equation}\label{char1a}
\chi_\kappa(z):= z^2 - cz + \alpha_\kappa + \beta_\kappa
e^{-chz}=0,  \ c >0,
\end{equation}
has  three real roots $\lambda_1\leq \lambda_2 <  0 < \lambda_3$
if and only if $c \leq c^{\frak L}_\kappa$.  If $c^{\frak L}
_\kappa$ is finite and $c=c^{\frak L} _\kappa$,  then equation (\ref{char1a}) has a double root $\lambda_1= \lambda_2<0$,
while for $c > c^{\frak L} _\kappa$ there does not exist any negative
root to (\ref{char1a}).  Moreover,  if $\lambda_j  \in \C$ is a
complex root of (\ref{char1a}) for $c \in (0, c^{\frak
L}_\kappa]$  then $\Re \lambda_j < \lambda_2$.

Furthermore,   $c^{\frak L}_\kappa(0)= +\infty$ and $c^{\frak L}
_\kappa(h)$ is strictly decreasing in its domain. In
fact, \[c^{\frak L}_\kappa(h)=
\frac{\theta(\alpha_\kappa,\beta_\kappa) +o(1)}{h}, \quad h \to
+\infty, \quad \mbox{ where} \
\theta(\alpha_\kappa,\beta_\kappa):=
\sqrt{\frac{2\omega_\kappa}{\beta_\kappa}}e^{\omega_\kappa/2},
\]
and $\omega_\kappa$ is the unique negative root of
\begin{equation}\label{roe}
-2\alpha_\kappa = \beta_\kappa e^{-\omega_\kappa}(2+\omega_\kappa).
\end{equation}
\end{lemma}

\begin{lemma} \label{lc1} Given $\alpha_0+ \beta_0 >0, \ \beta_0 \geq 0,$ there exists $c^{\frak L}_0=c^{\frak L}_0(h)>0$ such that the characteristic equation
\begin{equation}\label{char1}
\chi_0(z):= z^2 - cz +\alpha_0 +\beta_0 e^{-chz}=0,  \ c >0,
\end{equation}
has exactly two simple real roots $\lambda = \lambda(c), \mu=
\mu(c)$ if and only if $c > c^{\frak L}_0$. These roots are
positive so that  we can suppose that $0 < \lambda < \mu$.  Next,
if $c > c^{\frak L}_0$ and $\beta_0  >0$, then  all complex roots
$\{\lambda_j\}_{j \geq 1}$ of (\ref{char1}) are simple and can be
ordered in such a way that
\begin{equation}\label{por1}
\dots \leq\Re \lambda_3(c) \leq \Re \lambda_4(c)\leq  \Re
\lambda_2(c) = \Re \lambda_1(c) < \lambda < \mu.
\end{equation}
If $c=c^{\frak L}_0$, then the above equation has a double
positive root $\lambda(c^{\frak L}_0)= \mu(c^{\frak L}_0)$, while
for $c < c^{\frak L}_0$ there does not exist any real root to
(\ref{char1}).  Furthermore,  each complex root $z_0 = x_0 + i
y_0$ with $\Re z_0 = x_0 \leq \lambda(c)$ must have its imaginary
part $|\Im z_0| > \pi/ch$.  Finally, $c^{\frak L}_0=c^{\frak
L}_0(h)>0$ is a decreasing function, with $c^{\frak
L}_0(+\infty)=0$ if $\alpha_0 \leq 0$ and $c^{\frak
L}_0(+\infty)=2 \sqrt{\alpha_0}$ if $\alpha_0 > 0$.  In fact, for
$\alpha_0 \leq 0$, we have \[c^{\frak L}_0(h)=
\frac{\theta_1(\alpha_0, \beta_0) +o(1)}{h}, \quad h \to +\infty,
\ \mbox{ where} \quad \theta_1(\alpha_0, \beta_0):=
\sqrt{\frac{2\omega_0}{\beta_0}}e^{\omega_0/2},
\]
and $\omega_0$ is the unique positive root of
$
-2\alpha_0=\beta_0e^{-\omega_0}(2+\omega_0).
$
\end{lemma}
\begin{lemma} \label{lc3}  Assume that all conditions of Lemmas \ref{lc2} and \ref{lc1} are satisfied. Then equation $c^{\frak L}_\kappa(h)= c^{\frak L}_0(h)$  has exactly one non-negative solution $h_0$ if
$\theta(\alpha_\kappa,\beta_\kappa) < \theta_1(\alpha_0,\beta_0)$
and does not have any non-negative solution otherwise.
\end{lemma}
\begin{corolary} Set
$ \mathcal{D}_{\frak L} = \{(h,c): h \geq 0, \   c^{\frak L}_0(h)
\leq c \leq c^{\frak L} _\kappa(h)\}\cap \R^2 \subset \R_+^2. $
Then $\mathcal{D}_{\frak L}$ is a connected closed domain
containing $\{0\} \times [c^{\frak L}_0(0), +\infty)$.
\end{corolary}
Figure 2 below presents two possible forms of
$\mathcal{D}_{\frak L}$, in the second case
$\theta_1(\alpha_0, \beta_0) < \theta(\alpha_\kappa,
\beta_\kappa)$.

Next, let $\phi $ be a strictly monotone wavefront of (\ref{e1}).
The characteristic exponents $\Lambda_\pm$ of $\phi$ are defined
as $ \Lambda_\pm(\phi):= \lim_{t \to \pm \infty}  (1/t)\ln
|\phi(\pm\infty) - \phi(t)|. $
\begin{definition}\label{lpm}
 Let $\mathcal{D}_{\frak N}$ stand for  the maximal  connected open (in topology of $\R_+^2$)
component of  the set \[ \{(h,c) \in \mathcal{D}_{\frak L} :
\Lambda_-(\phi) = \lambda(c), \Lambda_+(\phi) = \lambda_2(c)\ {\rm
for\ each \ monotone \ wavefront \ } \phi \}
\]
which has a non-empty  intersection (cf. Lemma \ref{vt3}) with the vertical line
$h=0$. \end{definition} In general, description of
$\mathcal{D}_{\frak N}$ is  a very difficult task, related to the
determination of the minimal speed of propagation \cite{TPT}. But when the nonlinearity $f$ 
is dominated by its linearizations at the equilibria $0$ and $\kappa$, this task can be easily accomplished:
\begin{lemma}\label{tdo}
$\overline{\mathcal{D}_{\frak{N}}}$ coincides with
$\mathcal{D}_{\frak{L}}$ for each  $f(x,y) \in C^{1,\gamma}$
satisfying
\[
 f(x,y) \leq \alpha_0x +\beta_0 y, \quad   f(x,y) \leq \alpha_\kappa(x-\kappa)+\beta_k(y-\kappa), \ (x,y) \in [0, \kappa]^2.
\]
\end{lemma}
In other cases, we can still indicate explicitly a substantial subset of
 $\mathcal{D}_{\frak N}$, see Section 2.

In the sequel, we  will consider the following sign/monotonicity assumptions:\\
\begin{description}
\item[ ({\bf M})]  Each profile  $\phi: \R \to (0,\kappa)$ of
travelling front to (\ref{e1})  is a monotone function. \item[{\bf (MG)}]
$\alpha_0+\beta_0 > 0,\ \alpha_0 <0,\ \beta_0 >0,\ \alpha_\kappa <
0,\ \beta_\kappa < 0,$ and for each strictly increasing  $\zeta\in
C^2(\R),$ $\zeta(-\infty)=0, \ \zeta(+\infty) = \kappa$, it holds
that $f_1(\zeta(t),\zeta(t-ch)) \leq 0,$ $ t \in \R$,  while
$f_2(\zeta(t),\zeta(t-ch))$ has a unique zero on  $\R$. \item[{\bf
(KPP)}] $\beta_0 = 0, \alpha_0 > 0, \alpha_\kappa = 0,
\beta_\kappa < 0, $ and for each strictly increasing
$C^2$-function $\zeta= \zeta(t), \ \zeta(-\infty)=0, \
\zeta(+\infty) = \kappa$, it holds that $\alpha_0 \geq
f_1(\zeta(t),\zeta(t-ch)) \geq 0,$ $ t \in \R$,  while
$f_2(\zeta(t),\zeta(t-ch)) \leq 0$ on  $\R$.
\end{description}

\vspace{3mm}

Now we are in position to state the main result of this work:
\begin{theorem}  \label{ngz}Assume that  either hypotheses $({\bf M})\& ({\bf MG})$ or  $({\bf M})\&({\bf KPP})$ are satisfied. Then there is a global family $\mathcal{F} =
\{\phi(\cdot,h,c),  (h,c) \in \overline{\mathcal{D}_{\frak N}}\}$
of monotone travelling fronts to (\ref{pe}). Moreover, if $u =
\phi(\nu \cdot x+ct)$ is an eventually monotone front to
(\ref{pe}), then $(h,c) \in \mathcal{D}_{\frak L}$.
\end{theorem}
\begin{remark} a) In consequence, if $\overline{\mathcal{D}_{\frak N}}= \mathcal{D}_{\frak L}$ then Theorem \ref{ngz} provides a criterion of existence of monotone wavefronts. Moreover, what is quite important for applications, this criterion can be formulated explicitly   (in terms of coefficients of the characteristic equations (\ref{char1}), (\ref{char1a}), see Section 2).  b) Theorem 1.4 in \cite{TPT} suggests that $\overline{\mathcal{D}_{\frak N}}$ might
be the maximal domain of the monotone fronts existence even when $\overline{\mathcal{D}_{\frak N}}\not= \mathcal{D}_{\frak L}$. In particular, this would imply that $c^{\frak N}_*(h')= \inf \{c: (h',c) \in  \mathcal{D}_{\frak N} \}$ and that Theorem \ref{ngz}  yields  an existence criterion even when $\overline{\mathcal{D}_{\frak N}}\not= \mathcal{D}_{\frak L}$. In any case,  as we have already mentioned,
the explicit determination of $c^{\frak N}_*$ (and, in consequence, of $\overline{\mathcal{D}_{\frak N}}$)  is a very difficult problem even for non-delayed equations. \quad \quad \qquad c) As we will show, the family of all monotone wavefronts has the following property of local continuity: if $(h',c') \in \mathcal{D}_{\frak N}$ then there exists an open neighborhood $\mathcal U\subset \R^2_+$ of $(h',c')$ and a local family of monotone fronts $\phi_U$ such that $\phi_U(\cdot,h,c)$ depends continuously  on $(h,c) \in \mathcal U$ in the metric of weighted uniform convergence on $\R$.
\end{remark}
Finally, a few words about the organization of the paper.
Theorem \ref{ngz} is proved in Sections 3 and 4, while in the next
section it is applied to two important families of delayed
diffusion equations. Appendix to this paper contains the proofs of
all four lemmas announced in the introduction.
\section{Applications}\label{apa}
\subsection{The KPP type delayed equations}
Recently, a criterion of existence  of monotone fronts for the KPP-Fisher
equation
\begin{equation}\label{kpp}
u_t(t,x)=\Delta u(t,x)+u(t,x)(1-u(t-h,x))
\end{equation}
was established   in \cite{KO} by means of the shooting
techniques and in \cite{GT} by using a constructive monotone
iteration algorithm.  In this section, we apply Theorem \ref{ngz}
to a broad family of equations (\ref{pe}) which contains
(\ref{kpp}) as a particular case. It is worth to mention  that the
monotone wavefronts of the KPP-Fisher delayed  equation
(\ref{kpp}) have an additional nice property: they are {\it
absolutely} unique \cite{FW,GT,HT}. Thus the  family $\mathcal{F}
= \{\phi(\cdot,h,c),  (h,c) \in \mathcal{D}_{\frak L}\}$ of
monotone wavefronts to (\ref{kpp}) is actually globally
continuous.

We will say that monostable nonlinearity $f(x,y)$ in (\ref{pe}) is
of the KPP type, if 
\[ f \in C^{1,\gamma} \ \mbox{for some} \ 
\gamma \in\  (0,1],\quad 
\alpha_0>0, \  \beta_0=0, \ \alpha_\kappa = 0, \  \beta_\kappa
<0,\   \ f(0,y) \equiv 0,   \]
\[ \mbox{and, for all} \   x,y \in
(0,\kappa),
0< f_1(x,y) \leq \alpha_0, \ f_2(x,y) \leq 0,  \  0 < f(x,y) \leq
\beta_k(y-\kappa).
\]
It is then easy to see that  the set $\mathcal{D}_{\frak{L}}$ has
the  form given on Fig. 1, cf. \cite{GT}. 
\begin{figure}[h]
\centering \fbox{\includegraphics[width=5.5cm]{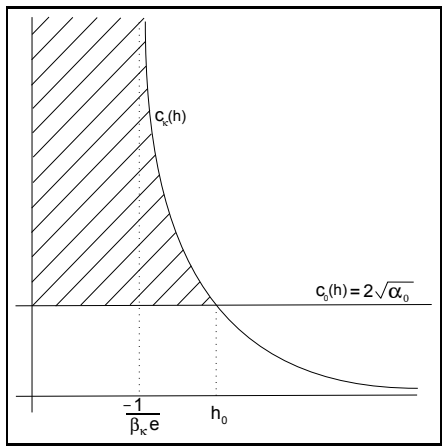}}
\caption{\hspace{0.5cm} Domain $\mathcal{D}_{\frak{L}}$ for the
KPP type delayed equation}
\end{figure}

In fact,  $c=c_{\kappa}(h)$ can be found from the equation
\[2+\sqrt{c^4h^2+4}=-\beta_\kappa c^2h^2\exp\left({1+\frac{2}{c^2h+
\sqrt{c^4h^2+4}}}\right).\] This allows to calculate easily $h_0$ (defined in Lemma \ref{lc3}) and the asymptote $h=-1/(e\beta_\kappa)$. 

\vspace{-2mm}

\begin{theorem}\label{KPPT}
Let $f$ be of the KPP type.  Then there is a monotone
 front $u =
\phi(\nu \cdot x+ct), $ $ |\nu|=1, \ c>0$,  to (\ref{kpp}) if and only if $(h,c) \in \mathcal{D}_{\frak{L}}$.
\end{theorem}

\vspace{-6mm}

\begin{proof} Since, by Lemma \ref{tdo},
$\overline{\mathcal{D}_{\frak{N}}}=\mathcal{D}_{\frak{L}}$, we
have only to check that the  hypotheses  $({\bf M})$ and $({\bf
KPP})$ are satisfied. 
First, 
$\bf{(KPP)}$ clearly holds due to the above definition of the KPP type nonlinearity.  Next, suppose for a moment that $\phi'(t_0)=0$ at
some $t_0\in \R$. Since 
$\phi(t_0),$ $\phi(t_0-ch) \in (0, \kappa)$, we have  $\phi''(t_0)=-f(\phi(t_0),
\phi(t_0-ch))<0$ and therefore $t_0$ is the only critical point of
$\phi$ (strict local maximum),  in contradiction with the boundary conditions at $\pm \infty$.  Thus  we have
$\phi'(t)>0$ for all $t$.
\end{proof}
\subsection{The Mackey-Glass type delayed diffusion equations}
Consider the following monostable equation
\begin{equation}\label{EDPMckG}
u_t(t,x)=\Delta u(t,x)-\delta u(t,x)+g(u(t-h,x)),
\end{equation}
where $C^{1,\gamma}-$continuous $g:\R_+\to \R_+ , \ g(0)=0, \
g(\kappa)=\delta \kappa, \ g'(0)>\delta >0,$ has a unique critical
point (a global maximum) on $(0,\kappa)$.  Clearly, 
$\alpha_0=\alpha_\kappa = -\delta, \ \beta_0=g'(0)>\delta$ and
$\beta_\kappa=g'(\kappa)< 0$.

\begin{theorem} \label{MGth} Let $g$ satisfy the above conditions.
Then there exists a family of monotone
wavefronts $u:=\phi(x\cdot \nu+ct, h,c), \ |\nu|=1, \ c >0,$ parametrized by $(h,c)\in
\overline{\frak{D}}_{\frak{N}}$.
\end{theorem}

\begin{proof} Observe that the monotonicity assumption $(\bf{M})$ is satisfied in view
of  \cite[Theorem 1.1]{TTT}.
 In order to check $(\bf{MG})$, suppose that  $\zeta\in
C^2(\R,(0,\kappa))$ is a strictly increasing function such that
$\zeta(-\infty)=0$, $\zeta(+\infty)=\kappa$. Then
$f_1(\zeta(t),\zeta(t-ch))=-\delta < 0, \ t \in \R,$ while
$f_2(\zeta(t),\zeta(t-ch))=g'(\zeta(t-ch))$ clearly has a unique
zero on $\R$.
\end{proof}
\subsection{The diffusive Nicholson's equation}
 Equation (\ref{EDPMckG}) with $g(x) = pxe^{-x}$  is called the diffusive Nicholson's equation.
 It is monostable when $p/\delta >1$, with steady state solutions $u_1:= 0$ and $u_2:= \ln(p/\delta)$.
If, in addition,  $p/\delta \leq e$, then $g(x)$
is monotone on $[u_1,u_2]$ and therefore there exists a unique
monotone travelling front for each fixed $c\geq c^{\frak{L}}_0$,
cf.  \cite{AGT,TPT,z}. In fact, this front can be found as a limit
of a converging monotone functional sequence \cite{z}. The
uniqueness can be deduced either by using the Diekmann-Kaper
theory \cite{AGT} or by applying the sliding method of Berestycki
and Nirenberg \cite{TPT}.  Now, if $e<{p}/{\delta}\leq e^2$, then
travelling fronts exist for every fixed $h\geq 0$  and
$c\geq c^{\frak{L}}_0$ \cite{ma1,TT}. However, they are not monotone for large
$c$ and $h$  \cite{TT}. Finally, if $p/\delta > e^2$ then the
wavefronts exist only for $h$ from some bounded set (depending on
$p,\delta$) \cite{TT,TTT}. If $p/\delta > e^2$ and $h$ is large,  then the
Nicholson's equation possesses positive and bounded semi-wavefront
solutions, i.e. solutions $u= \phi(\nu\cdot x+ ct),\
\phi(-\infty)=0,\ \liminf_{t \to +\infty} \phi(t) >0$. It was also
proved in \cite{TTT} that, for $p/\delta \in (e^2, 16.99..)$,  these solutions have monotone leading
edge and that they are either eventually monotone or slowly
oscillating at $+\infty$, cf. Corollary \ref{c24} below. It is an open problem whether there can exist an eventually monotone and non-monotone front for some $p/\delta >e^2$.  Hence, excepting the above mentioned
result from \cite{z}, nothing was known about the existence of
monotone fronts for the Nicholson's equation. Our first assertion
here gives the complete solution to the considered problem for
$p/\delta \leq e^2$:
\begin{theorem} Assume that   $g(x) = pxe^{-x}$   and  $p/\delta \in (e, e^2]$. Then equation  (\ref{EDPMckG})  has a unique (up to a translation) travelling front for each $c \geq c_0^{\frak L}, \ h \geq 0$. This front is monotone if and only if $(h,c) \in \mathcal{D}_{\frak L}$.  The domain
$ \mathcal{D}_{\frak L}$ has two main geometric forms presented on
Fig. 2, where $\nu_0 := 2.808 \dots $ and $\delta h_ae^{\delta
h_a}=(e\ln(p/e\delta))^{-1}$.
\end{theorem}
\begin{figure}[h]
\centering \fbox{\includegraphics[width=5cm]{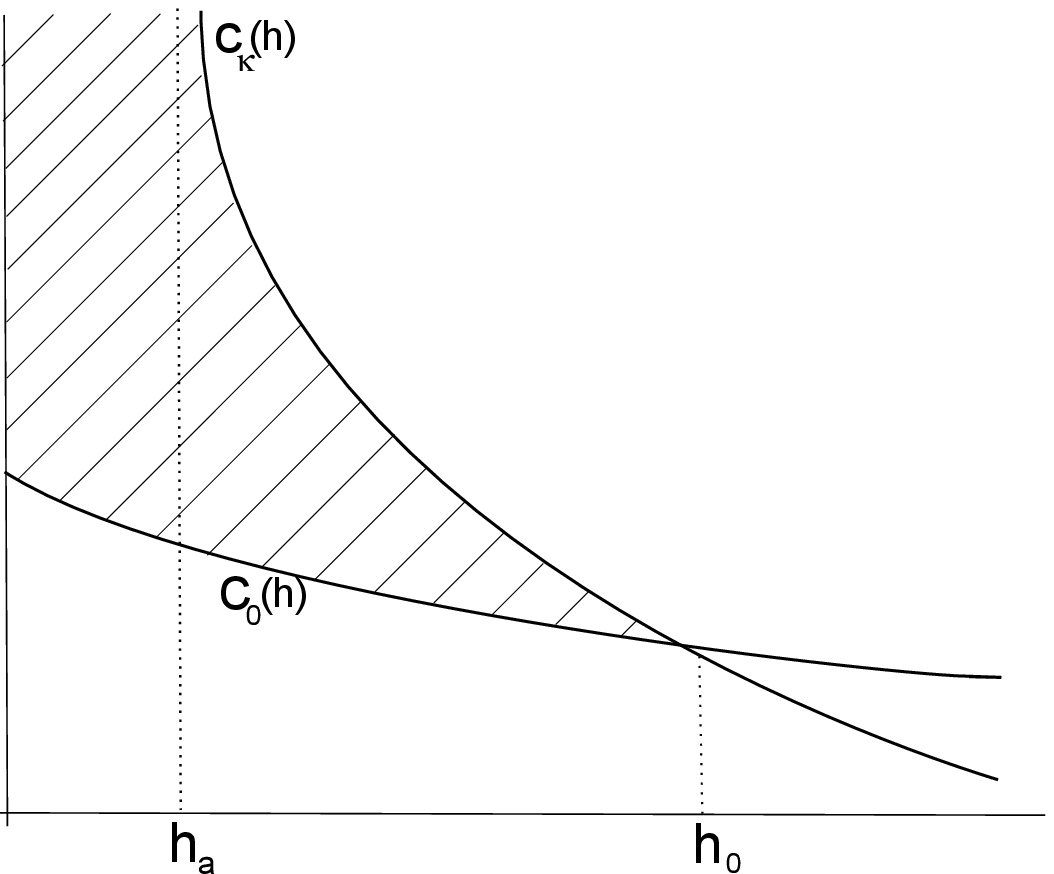}}
\fbox{\includegraphics[width=5cm]{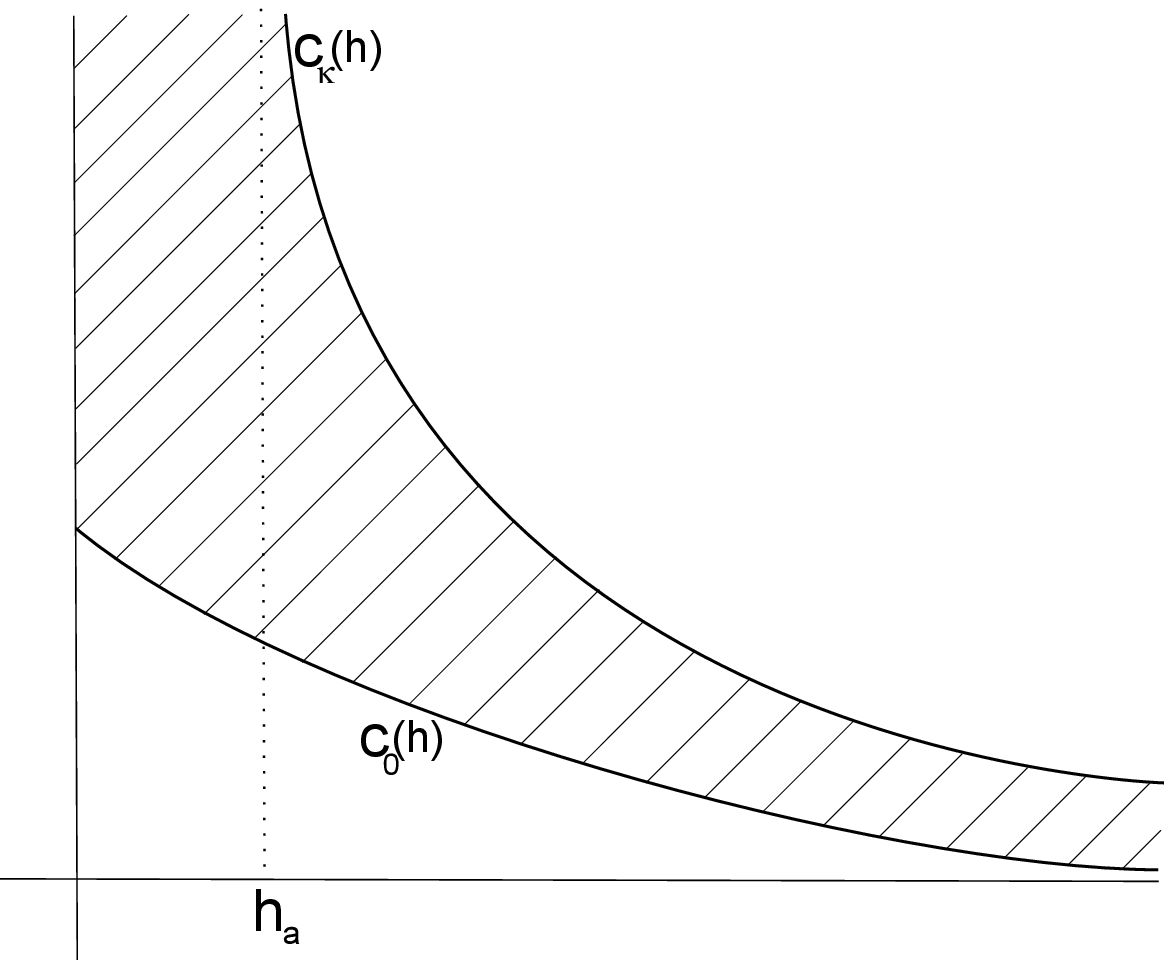}}
\caption{\hspace{0.5cm}$\nu_0<p/\delta$ \hspace{3cm} $\nu_0\geq
p/\delta$}
\end{figure}

\vspace{-6mm}

\begin{proof} The front existence for $c \geq c_0^{\frak L}$ was proved in \cite{ma1,TT}.  The uniqueness statement follows from \cite{AGT}.  If $p/\delta \in (e, e^2]$ then $\min_{[u_1,u_2]} g'(x) = g'(u_2)$ and therefore Lemma \ref{tdo} assures that $\overline{\mathcal{D}_{\frak{N}}}=\mathcal{D}_{\frak{L}}$.
Hence, in order to prove our criterion for the existence of
monotone fronts, it suffices to  invoke Theorem \ref{MGth}.

Next,  $\alpha_0=-\delta$, $\beta_0=p$, $\alpha_\kappa=-\delta$
and $\beta_\kappa=\delta \ln(e\delta/p)$. As a consequence,
functions $c=c_0^{\frak L}(h)$ and $c=c_\kappa^{\frak L} (h)$ are
determined, respectively, by the equations
\[
\frac{c^2+4\delta}{2+\sqrt{c^4h^2+4c^2h^2\delta+4}}=ep\exp\left(-\frac{\sqrt{c^4h^2+4c^2h^2\delta+4}+c^2h}{2}\right),
\ h \geq 0;
\]
\begin{equation}\label{ck+}
\hspace{-7mm}
\frac{2+\sqrt{c^4h^2+4c^2h^2\delta+4}}{ec^2h^2|\beta_\kappa|}=\exp\left(\frac{\sqrt{c^4h^2+4c^2h^2\delta+4}-c^2h}{2}\right),
h > h_a,
\end{equation}
where $h_a$ is such that  $e|\beta_\kappa| h_a\exp({\delta
h_a})=1$. A simple  analysis shows  that $c_\kappa^{\frak L} (h)=
+\infty$ if and only if $h\in [0,h_a]$.  Next, $\theta_1(\alpha_0,
\beta_0)=\sqrt{\frac{2w_0}{p}}e^{w_0/2}$ where $w_0$ is the
positive root of $2\delta/p=e^{-w}(2+w)$ (see Lemma \ref{lc1}).
Similarly, from Lemma \ref{lc2} we infer that
$\theta(\alpha_\kappa,\beta_\kappa)=\sqrt{\frac{2|w_0|}{\delta\ln(p/e\delta)}}e^{w_0/2}$,
where $w_0$ is the negative root of
$-2/\ln(p/e\delta)=e^{-w}(2+w)$. By  Lemma \ref{lc3}, the value of
$\nu_0= p/\delta$ is determined by the condition
$\theta(\alpha_\kappa, \beta_\kappa)=\theta_1(\alpha_0, \beta_0)$.
It is easy to show that
$\nu_0=\frac{p}{\delta}(t_0)=t_0^{-1}(-1+\sqrt{1+2t_0})e^{-1+\sqrt{1+2t_0}}$
with $t_0$ being  the positive root of
\[t_0^{-1}(-1+\sqrt{1+2t_0})\exp\left(-2+\sqrt{1+2t_0}\right)=\exp\left(t_0^{-1}(1+\sqrt{1+2t_0})e^{-1-\sqrt{1+2t_0}}\right).\]
Finally, we find   $\nu_0=2.808\dots \in (e,e^2)$.
\end{proof}

\vspace{-6mm}

\begin{corolary} \label{c24}Suppose that $p/\delta \in (2.718\dots, 2.808\dots]$, then each minimal wavefront is monotone (independently on $h$). If  $p/\delta \in  (2.808\dots, 16.99904\dots]$ and $h > h_0$, then every  minimal wavefront is slowly oscillating at $+\infty$.
\end{corolary}

\vspace{-6mm}

\begin{proof}  If $p/\delta \leq  2.80\dots$, then the domain $\overline{\mathcal{D}_{\frak{N}}}=\mathcal{D}_{\frak{L}}$ is unbounded from the right (see Fig. 2) and the first statement follows. If  $p/\delta \in  (2.80\dots, 16.99\dots]$  then the positive feedback assumption of \cite[Theorem 3]{TTT} is satisfied and therefore each wavefront is either eventually monotone or slowly oscillating.  However, if $h >h_0$, then none front solution can be eventually monotone due to Theorem \ref{ngz}.
 \end{proof}
Let now $p/\delta > e^2$.  Then $\beta_\kappa^-:= \inf_{x\in
(0,u_2)}(g(x)-g(u_2))/(x-u_2) <0$ and
 $g(x) \leq \beta_\kappa^- (x-u_2)  + g(u_2), \ x \in [0,u_2]$.  We also will need function $c:= c^-_\kappa(h)$ which is implicitly (and analogously to $c_\kappa$) defined by equation (\ref{ck+}) where $\beta_\kappa,  h_a$ are replaced with  $\beta_\kappa^-$ and $h_a^-$ (such that  $e|\beta_\kappa^- | h_a^-\exp({\delta h_a^-})=1$), respectively. In particular,
 $ c^-_\kappa(h):= +\infty$ for $h \in [0,h_a^-]$.
 It is easy to see that
 $0< h_a^- \leq h_a$ and that $c^-_\kappa(h)\leq c_\kappa(h), h \in [0,h_0^-]$. Here $h_0^-$ satisfies
 $c^-_\kappa(h_0^-) = c_0^{\frak L}(h_0^-)$.
\begin{theorem} \label{TR} Suppose  that  $p/\delta  > e^2$ and $c \in [c_0^{\frak{L}}(h), c^-_\kappa(h)], \ h  \leq h^-_0$.
Then the Nicholson's equation   has a unique (up to a translation)
monotone wavefront.
\end{theorem}
\begin{proof} This result follows from Theorem \ref{MGth} if we observe that  $\Int\,{\frak{D}}^-_{\frak{L}}:= \{(h,c): c \in (c_0^{\frak{L}}(h), c^-_\kappa(h)), \ h  \in [0,h^-_0]\} \subset \overline{\frak{D}}_{\frak{N}}$ (this inclusion  is justified in Appendix,  Remark \ref{ru}).  The front uniqueness is due to the relation
$g'(0)= \max_{s\geq 0}|g'(s)|$, e.g. see  \cite{AGT}.
\end{proof}
\section{Associated Fredholm operator}
Let $\phi$ be a monotone solution of equation (\ref{e1})
connecting equilibria $0$ and $\kappa$. The spectra  of the
linearization of (\ref{e1})  at $0,  \kappa$ were analyzed in
Lemmas \ref{lc1}, \ref{lc2}.  In this section, we study  the
linear variational equation along the solution $\phi$
\[
 v''(t)-cv'(t) + f_1(\phi(t),\phi(t-ch))v(t) + f_2(\phi(t),\phi(t-ch))v(t-ch)=0.
\]
With the notation $ a(t):=  f_1(\phi(t),\phi(t-ch)), \ b(t):=
f_2(\phi(t),\phi(t-ch)), $ this equation can be written as the
system
\begin{equation}\label{asy}
 v'(t)  = w(t), \
  w'(t) = -a(t)v(t)+ c w(t) - b(t)v(t-ch),
\end{equation}
or shortly as ${\frak F}_c(v,w)=0$, where \[ {\frak F}_c (v,w)(t)
= (v'(t)-w(t), w'(t)+  a(t)v(t)- c w(t) + b(t)v(t-ch)).
\]
For small $\delta>0$ and fixed $c$, we define the following Banach
spaces:
\[
C_\delta = \{\psi \in C(\R, \R^2): |\psi|_\delta: = \sup_{s \leq
0}e^{-(\lambda(c)- \delta)s} |\psi(s)|+  \sup_{s \geq
0}e^{-(\lambda_{2}(c) + \delta)s} |\psi(s)| < \infty  \},
\]
\[
C^1_\delta = \{\psi \in C_\delta: \psi, \psi' \in C_\delta, \
|\psi|_{1,\delta}: =  |\psi|_\delta +  |\psi' |_\delta < +\infty
\}. \] We will consider ${\frak F}_c$ as a linear operator defined
on $C^1_\delta$ and taking its values in $C_\delta$. The main
result of this section follows:
\begin{theorem}\label{thIF}Let either ({\bf MG}) or ({\bf KPP}) hold with  $\zeta(t) = \phi(t)$. If $(h,c)\in \Int\,\mathcal{D}_{\frak{L}}$ then ${\frak F}_c: C^1_\delta  \to C_\delta$ is a surjective
Fredholm operator, with $\dim Ker \,{\frak F}_c=1$.
\end{theorem}
We will prove this theorem by using Hale and Lin  analysis
\cite[Lemmas 4.5-4.6]{HL} of the linear functional differential
equations
\begin{equation}\label{hlfde}
y'(t)=L(t)y_t, \quad  y_t(s) := y(t+s),   \quad L(t): C([-ch,0],
\R^n) \to \R^n,
\end{equation}
where linear bounded operators $L(t)$ depend  continuously  on $t
\in \R$ in the operator norm and are uniformly bounded on $\R$.  Let $Y(t,s)$ denote
the evolution (solution) operator for (\ref{hlfde}). Then the
equation is said \cite{HL} to have a shifted  exponential
dichotomy   on a half-line $J$ with the exponents $\alpha < \beta$
and projection $P_u(s), s \in J,$ if \[|Y(t,s)(I-P_u(s))| \leq
Ke^{\alpha(t-s)}, \quad |Y(t,s)P_u(s)| \leq Ke^{\beta(t-s)},
\quad t \geq s \in J.\] Take some $\nu \in (\alpha, \beta)$ and
consider the change of variables $y(t) = x(t)e^{\nu t}$ which
transforms  (\ref{hlfde}) into $x'(t) = M(t)x_t$ with
$M(t)\phi(\cdot) = L(t)(e^{\nu \cdot}\phi(\cdot)) - \nu\phi(0)$ and
the evolution operator $X(t,s) = e^{-\nu(t-s)} e^{-\nu \cdot}
Y(t,s) e^{\nu \cdot} $. It is clear that the transformed equation
has a usual exponential dichotomy with the exponents $\alpha-\nu<0
<\beta -\nu,$ and projection $e^{-\nu \cdot}
P_u(s) e^{\nu \cdot}, s \in J,$
if and only if
the original equation (\ref{hlfde}) has a shifted  exponential
dichotomy  with the exponents $\alpha < \beta$ and 
projection 
$P_u(s), s \in J$.

For convenience of the reader, in Proposition \ref{HaLi}  below we
summarize  the content of the mentioned  lemmas from \cite{HL} for
the special case of system (\ref{asy}) whose formal adjoint equation \cite{Hale} is
given by
\begin{equation}\label{asya}
 y_1'(t)  = a(t)y_2(t)+b(t+ch)y_2(t+ch), \
  y_2'(t) = -y_1(t)-cy_2(t).
\end{equation}
Particular solutions $y=(y_1,y_2)$ of  (\ref{asya}) which are
defined on $\R$ and satisfy
\begin{equation} \label{fae}
|y(t)| \leq Ke^{-\beta_2 t}, \ t \geq 0, \quad |y(t)| \leq
Ke^{-\alpha_1 t}, \ t \leq 0,
\end{equation}
for some $K, \alpha_1, \beta_2$ (specified below) will be of
special importance:
\begin{proposition} \label{HaLi} Suppose that continuous functions $a,b: \R \to \R$ are bounded and, for some $\tau >0$,  system (\ref{asy}) has shifted dichotomies in $(-\infty,-\tau]$ and $[\tau,+\infty)$ with exponents $\alpha_1 = \lambda(c) - \delta < \beta_1$, $\alpha_2 < \lambda_2(c) +\delta < \beta_2$ and projections $P_u^-(t), P_u^+(t)$, respectively.  Then \qquad \qquad ${\frak F}_c: C^1_\delta  \to C_\delta$  is Fredholm of index $
i({\frak F}_c)= \dim {\mathcal R}P_u^-(-\tau) - \dim  {\mathcal
R}P_u^+(\tau), $  and with the range  \[
  {\mathcal R}({\frak F}_c)= \{h \in C_\delta: \int_\R y(s)h(s)ds=0 \ \textrm{for all solutions} \ y(t) \ \textrm{of} \ (\ref{asya}) \ \textrm{satisfying} \ (\ref{fae})\}.
\]
\end{proposition}
Now, since system (\ref{asy}) is asymptotically autonomous and
the eigenvalues $\lambda(c), \lambda_2(c)$ of the limit systems
for (\ref{asy})  at $\pm \infty$ are real and  isolated,  the
roughness property of the exponential dichotomy (cf. \cite[Lemma
4.3]{HL}) implies the following.  For sufficiently large $\tau
>0$,  system (\ref{asy}) has shifted dichotomies in
$(-\infty,\tau]$ and $[\tau,+\infty)$ with exponents,
respectively, \[\alpha_1 = \lambda(c) - \delta >0, \ \beta_1=
\lambda(c) - \delta/2, \ \alpha_2 := \lambda_2(c) + \delta/2 <0, \
\beta_2: = \delta, \]
\[
\mbox{and} \  \dim {\mathcal R}P_u^-(-\tau) = 2, \  \dim
{\mathcal R}P_u^+(\tau) = 1,  \ \textrm{so that } \ i({\frak F})=
1.
\]
Let  $(h,c)\in \Int\,\mathcal{D}_{\frak{L}}$, then, for each wavefront $\phi$, we have $(\phi, \phi') \in C^1_\delta$ (cf. Remark \ref{e1r}) and
${\frak F}_c(\phi, \phi') (t) =0$. As a consequence,  $\dim$
Ker${({\frak F}_c)} \geq 1$. Theorem \ref{thIF} claims that
actually $\dim$ Ker${{\frak F}_c} = 1$ because of  co$\dim{\mathcal
R}{{\frak F}_c} = 0$.  In order to prove that
 ${\mathcal R}({\frak F}_c)= C_\delta$  it suffices to show that none nontrivial solution  of  (\ref{asya}) can satisfy (\ref{fae}).  We establish this fact in the next lemmas.

At this stage, it is worth rewriting (\ref{asya}) and (\ref{fae})
in a more familiar way. First, we observe that (\ref{asya}) 
reduces to the second order equation
\[
y''(t) = -cy'(t) - a(t)y(t)- b(t+ch)y(t+ch).
\]
Next, after the change of variables $v(t)= y(-t), \ t \in \R,$ we
obtain that
\[
v''(t) -c v'(t) +a(-t)v(t) + b(-t+ch)v(t-ch) =0,
\]
while inequalities (\ref{fae}) take the form
\begin{equation}\label{2inw}
|v(t)|+|v'(t)| \leq Ke^{\delta t}, \ t \leq 0, \quad |v(t)| +
|v'(t)|\leq Ke^{( \lambda(c) - \delta) t}, \ t \geq 0.
\end{equation}
Set $A(t) := a(-t), \ B(t):= b(-t+ch)$.  It is clear that $A,B$ are
continuous  with
\[
A(-\infty)= \alpha_\kappa, \ B(-\infty) = \beta_\kappa, \
A(+\infty)= \alpha_0, \ B(+\infty) = \beta_0.
\]

\begin{lemma} \label{minf}  Let $(h,c) \in {\rm Int}\, \mathcal{D}_{\frak L}$. Then there exists a unique (modulo a constant factor) nontrivial solution
$v(t)$ of equation
\begin{equation}\label{we22}
v''(t)-cv'(t) + A(t) v(t) + B(t) v(t-ch)=0,
\end{equation}
such that  $v(t), v'(t) \to 0$ as $t \to -\infty$. Moreover, we
can suppose that $v(t) > 0, $ $\ v'(t) >0$ for all sufficiently
large negative $t$ while $\lim_{t \to -\infty}{v'(t)/v(t)}=
\lambda_3$.
\end{lemma}
\begin{proof} Setting $C_2:= C([-ch,0], \R^2)$, we can present (\ref{we22})  as the  system
\begin{equation}\label{we22p}
\hspace{-2mm} v'(t) = w(t), \ w'(t) = cw(t)-A(t)v(t)-B(t)v(t-ch).
\end{equation}
Since $(h,c) \in {\rm Int}\, \mathcal{D}_{\frak L}$,   the limit  system  of (\ref{we22p}) at  $-\infty$
 is exponentially dichotomic with some projection
$P$. In fact, it possesses  one-dimensional unstable invariant
submanifold of $C_2$ generated by the element $(v,w)(s)=
(e^{\lambda_3 s}, \lambda_3 e^{\lambda_3 s}),$ $s \in [-ch,0]$.
Thus $P(v,w) = (v,w)$. Using the roughness property \cite[Lemma 4.3]{HL} of the exponential
dichotomy, we obtain that the
perturbed system (\ref{we22p}) is also dichotomic on some
interval $(-\infty, -\tau] \subset \R_-$ with the projection $P(t)$
such that $P(t) \to P, \ t \to - \infty$. Set $(v_t,w_t)=
P(t)(v,w)$, then $P(t)(v_t,w_t)= (v_t,w_t)$ and
\[
|(v,w)-(v_t,w_t)|_{C_2} = |(P(t)-P)(v,w)|_{C_2} \to 0, \ t \to -
\infty.
\]
As a consequence, $v_t(s)> 0, w_t(s) >0, \ s \in [-ch,0],$ for
all sufficiently large negative $t\leq -\tau_1 \leq -\tau$. Next,
it is clear that every bounded on $\R_-$ solution $(v(t),v'(t))$
of (\ref{we22p}) can be written as
\[
(v(t+s),v'(t+s)) = \lambda(t) (v_t(s), w_t(s)), \ t \leq -\tau_1,
\ s \in [-ch,0],
\]
for some continuous scalar function $\lambda:(-\infty,-\tau_1] \to
\R$. It is easy to see from (\ref{we22p}) that $\lambda(t_0) =0$
for some $t_0 \leq -\tau_1$ if and only if $\lambda(t) = 0, t \leq
-\tau_1$. Therefore components of each bounded solution
$(v(t),v'(t))$ of (\ref{we22p}) keep their sign on $(-\infty,
-\tau_1]$. Finally, we have that
$
\lim_{t \to -\infty}v'(t)/v(t)= \lim_{t \to
-\infty}w_t(0)/v_t(0)= w(0)/v(0) =\lambda_3.
$
\end{proof}
\begin{lemma} \label{lmac} Assume that  either hypothesis $({\bf MG})$ or  $({\bf KPP})$ is satisfied.   Let   $ (h,c) \in {\rm Int}\, \mathcal{D}_{\frak L} $ and $A(t) = \alpha_0 + O(e^{-\gamma t}), \  B(t)= \beta_0+ O(e^{-\gamma t}), \ t \to +\infty, $
for some  $\gamma >0$. Then only the trivial solution $v(t) \equiv
0$ of  equation (\ref{we22}) can satisfy inequalities (\ref{2inw}). 
\end{lemma}
\begin{proof} Assume, on the contrary, that there is a nontrivial $v(t)$ satisfying  (\ref{2inw}),  (\ref{we22}).    By Lemma \ref{minf}, we can suppose that $v(t),v'(t)>0$  on some maximal open interval $(-\infty, \sigma)$ and $v'(\sigma)=0$ (whenever $\sigma$ is finite).

{In the first part of the proof, we will assume additionally that hypothesis {\bf (MG)}} is satisfied.  Then the open set $Z_v:= \{t \in \R: v(t)\not=0\}$ is dense in $\R$.  Indeed, otherwise $v(t) \equiv 0$ on some non-degenerate interval $[r_1,r_2]$ so that, in virtue of
equation  (\ref{we22}),  $v(t) \equiv 0$ for  $t \in
[r_1-chj,r_2-chj], \ j\in \N$. This, however, contradicts to the inequality  $v(t)
>0,$ $t \leq \sigma$. Now, if $v(t)$ is not a small solution
(the latter means that $\lim_{t\to+\infty}v(t)e^{st}=0$ for every
$s\in \R$), we obtain from
 \cite[Proposition 7.2]{FA} and Lemma \ref{lc1} that
\begin{equation}\label{rev}
v(t)= Ce^{x_j t}(\cos(y_j t +\varphi_j)+ o(1)), \quad t \to
+\infty,
\end{equation}
for some  $C >0, \ \varphi_j \in \R, $ and complex $\lambda_j:=
x_j + iy_j, \ |y_j| > \pi/ch,\ x_j < \lambda(c),$ satisfying
(\ref{char1}). Therefore $v(t)$ oscillates on $\R_+$ and  $\sigma$
is finite.   Let $t_*$ denote the unique
zero of $B(t)$ on $\R$. Since $v''(\sigma) \leq 0,$ $ v'(\sigma) =0, v(\sigma)
>0, v(\sigma-ch)>0$, we obtain that $\sigma \geq t_*$ because of
\[
0=v''(\sigma)-cv'(\sigma) + A(\sigma) v(\sigma) + B(\sigma)
v(\sigma-ch) \leq B(\sigma) v(\sigma-ch).
\]
Hence $B(t) > 0,\ A(t) \leq 0$ on $(\sigma, +\infty)$ and
therefore the nonlinearity
$$
(N_0,N_1):= (w(t), cw(t)- A(t)v(t)-B(t)v(t-ch))
$$
satisfies the following feedback inequalities (with $\delta^* =-1$,
see \cite{MPS}) for $t \geq \sigma$:
\begin{equation}\label{fis}
\left\{
\begin{array}{ccc}
N_0(t,0,w)= w \geq 0 & {\rm if \ and \ only \ if \ } w \geq 0, \
  \\
N_1(t,v,0,v_t)=  - A(t)v-B(t)v_t \geq 0 & {\rm \ if \ } v \geq 0 \
{\rm and \ } \ \delta^* v_t \geq 0,\\
N_1(t,v,0,v_t)=  - A(t)v-B(t)v_t \leq 0&  {\rm \ if \ } v \leq 0  \
{\rm and \ } \ \delta^* v_t \leq 0.
\end{array}
\right.
\end{equation}
In the next stage of the proof, we make use of the discrete Lyapunov
functional $V^-(\phi)$  introduced by J. Mallet-Paret and G. Sell in \cite{MPS}.  
For the convenience of the reader,  below we adopt to our situation the definition of $V^-$  and a key result from 
\cite{MPS}  describing the monotonicity properties of $V^-(v_t), \ t
\geq \sigma$.   Let us introduce a new notation: $\mathbb{K} = [-h,0] \cup \{1\}$.
\begin{definition}  For any
$v \in C(\mathbb{K})\setminus\{0\}$ we define the number of sign
changes by $$\hspace{-1mm} {\rm sc}(v) = \sup\{k \geq 1:{\rm \it
there \ are \ } t_0 <
 \dots < t_k, \ t_j \in \K,  \ {\rm \it such \ that\ }
v(t_{i-1})v(t_{i}) <0 {\rm \ for \ }  i\geq 1\}. $$ We set ${\rm
sc}(v) =0$ if $v(s) \geq 0$ or  $v(s) \leq 0$ for $s \in
\mathbb{K}$. If $\varphi \in C^1[-ch, 0]$ is not identically zero,  we write $(\bar \varphi )(s) = \varphi(s)$ if 
$s \in [-ch,0]$, and $(\bar \varphi)(1) = \varphi'(0)$.  Then the  Lyapunov  functional $V^-:C^1[-ch,0]\setminus\{0\} \to \{1,3,5, \dots\}$ is defined by the relations: $V^-(\phi) = \mbox{sc}(\bar\phi)$ if $\mbox{sc}(\bar\phi)$ is odd or infinite;  $V^-(\phi) = \mbox{sc}(\bar\phi)+1$ if $\mbox{sc}(\bar\phi)$ is even. 
\end{definition}
\begin{proposition}(By \cite[Theorem 2.1]{MPS}). \label{MPP}Assume that the feedback inequalities (\ref{fis}) hold for $t \geq \sigma$. Let $v:[\sigma -ch, +\infty) \to \R$ be a nontrivial $C^1$-solution of equation  (\ref{we22}), and set $v_t(s):= v(t+s), $ $  s \in [-ch,0]$. 
Then the discrete
Lyapunov functional  $V^-(v_t)$ is a nonincreasing function of $t \geq \sigma$ as long as $v_t$ is not the zero function. 
\end{proposition}
Since $V^-(v_\sigma) =1$, Proposition \ref{MPP} assures that  $V^-(v_t)
=1$ for $t \geq \sigma$.  On the other hand, in view of  $|y_j| >
\pi/ch$ and representation (\ref{rev}), we find that $V^-(v_t)
\geq 3$ for all large positive $t$.  This contradiction
shows that  $v(t)$ must be a small solution.
We will analyze the following two alternative cases:

$i)$  $v(t) \geq 0$ for all $t$ from some maximal subinterval
$[\hat t, \infty) \subseteq [\sigma, \infty)$. Since \[-A(t) =
|A(t)| \leq \frak{b}_0:= \max_{t \geq \sigma}  |A(t)|, \quad  -
B(t) \leq \frak{b}_1=  0,\ t \geq \sigma, \] we can apply
\cite[Lemma 3.1.1]{HVL}, under Assumption 3.1.2 with $\gamma=-1$,
to conclude that $v\equiv 0$ on some interval $[t_\#,
\infty)\subset \R\setminus Z_v$, a contradiction.

$ii)$ $v(t)$ is oscillating on $[\sigma, \infty)$. Since we know that $V^-(v_t)
=1$ for $t \geq \sigma$,   the number of   sign changes
of $v_t$ on $[t-ch,t]$ is less than 1. This implies  the existence of
an infinite sequence $\{t_j\}_{j \geq 0}, \ t_{j+1}-t_j \geq ch$,
such that $v(t_j)=0$ and $v(t) > 0$ [respectively, $v(t) < 0$]
almost everywhere on each $(t_{2j},t_{2j+1})$ [respectively,
$(t_{2j+1},t_{2j+2})$].  Next,  the property $V^-(v_t)
=1,  \ t \geq \sigma,$ yields  additionally that
$v'(t) \geq 0$ a.e.  on $(t_{2j},t_{2j}+ch)$.  In consequence, 
\[
v''(t) = cv'(t) +|A(t)|v(t)+ B(t)|v(t-ch)|\geq 0 \ \mbox{a.e. on}
\  [t_{2j},t_{2j}+ch].
\]
Therefore $v'(t), v(t) > 0$ for all $t \in (t_{2j},t_{2j}+ch]$.
This shows that, in fact, $t_{2j+1}-t_{2j} > ch$ and there is a
rightmost $s_j \in (t_{2j}+ch,t_{2j+1})$ such that $v(s_j)=
\max_{u \in [t_{2j},t_{2j+1}]} v(u)$.  Since $v(+\infty)=0$, without
restricting the generality, in the sequel we can assume that
$t_{2j}, s_j$ are choosen in such a way that $0< v(s_j) \geq
|v(t)|, \ t \geq t_{2j}$ (otherwise, it suffices to consider
$ -v(t)$). 

Hence,  $
\max_{u\geq s_{j}-ch}|v(u)| \leq v(s_j), $  and   for every fixed $T \geq 0$ and $t \in [s_{j}-ch,s_j+T]$, it
holds
\begin{eqnarray}
\hspace{-2mm}&&|v'(t)| \leq v'(t_{2j}+ch)+ \max_{u \in [q_{j}, s_j+T]}|v'(u)| \leq |\int_{t_{2j}+ch}^{s_j}e^{c(t_{2j}+ch-s)}(A(s)v(s)+B(s)v(s-ch))ds| +\nonumber \\
\hspace{-2mm}&&  \max_{t \in [q_{j}, s_j+T]}  |\int_{t}^{s_j}e^{c(t-s)}(A(s)v(s)+B(s)v(s-ch))ds| <
4\frac{|\alpha_0|+\beta_0}{c}e^{cT}v(s_j), \nonumber
\end{eqnarray}
where $q_j: = \max\{s_{j}-ch, t_{2j}+ch\}$. Therefore, if we set $w_j(t):= v(t+s_j-ch)/v(s_j)$, we have that 
 $|w_j(t)| \leq 1,$ $  t
\geq 0,\ w_j(ch)=1, w'_j(ch) =0$, and, for every fixed $T >0$, 
\[
|w_j'(s)|\leq 4\frac{|\alpha_0|+\beta_0}{c}e^{cT}, \ s \in [0,T]. 
\]
As a consequence, after an application of the Arzela-Ascoli
theorem, we obtain that $w_j$ has a subsequence (we will use the
same notation $w_j$ for it) such that $w'_j(ch) = 0, \  \lim
w_j(t) = w_*(t), \ t \in \R_+,$ where the convergence is uniform on
compact subsets of $\R_+$. It is clear that continuous $w_*$
is bounded: $1 = \max_{t \geq 0} w_*(t)  = 
w_*(ch).$ Note that $w_j(t)$ satisfies
\[
w''(t)-cw'(t) + A_j(t) w(t) + B_j(t) w(t-ch)=0, \ t \in \R,
\]
where $A_j(t):= A(t+s_{j}-ch) \to \alpha_0, \  B_j(t) := B(t+s_{j}-ch)
\to \beta_0$ uniformly on $\R_+$. Thus
\[
w_j'(t) = w_j'(ch) + c(w_j(t)-w_j(ch)) - \int_{ch}^t(A_j(s) w_j(s)
+ B_j(s) w_j(s-ch))ds
\]
converges (uniformly on compact subsets of $[ch, +\infty)$) to
$w'_*(t)$ and
\[
w'_*(t)=  c(w_*(t)-w_*(ch)) - \int_{ch}^t(\alpha_0
w_*(s) + \beta_0 w_*(s-ch))ds,
 \ t \geq ch.
\]
Thus $w_*(t)$ is a bounded solution of the linear delay
differential equation  (\ref{ve1}, $j=0$) considered for $t \geq
ch$, with non-negative initial value $w_*(s), \ s
\in [0,ch],$ and $w_*'(ch)= 0,$  $w_*(ch)= 1$. In view of (\ref{ve1}, $j=0$), this implies that
$w_*(t)\not\equiv 0$ on every subinterval $[p, +\infty), \ p \geq ch$.  By
\cite[Theorem 3.1, p. 76]{Hale} the latter assures that $w_*(t)$
is not a small solution of  (\ref{ve1}, $j=0$).  Moreover, since
(\ref{ve1}, $j=0$) satisfies the feedback assumptions  similar to (\ref{fis}) 
and $V^-(w_{*ch})=1$, Proposition \ref{MPP} implies 
$V^-(w_{*t}) =1$ for $t \geq ch$.  However, invoking again representation (\ref{rev}),  
we find that $V^-(w_{*t}) \geq 3$ for all large positive $t$, a contradiction.

{Assume now condition ({\bf KPP})}. By Lemma \ref{minf},
without restricting the generality, we can suppose that $0 \in
(-\infty, \sigma)$ and  $v'(0)/v(0) \approx \lambda_3$. Let
$(-\infty,\sigma_*)$ denote the maximal open interval where
$v(t)>0$ (it is clear that  $\sigma_* \geq  \sigma$). Observe  that
\[
v''(t)-cv'(t) + \alpha_0 v(t) = D(t), \ {\rm where \ } D(t):=
(\alpha_0-A(t))v(t)-B(t)v(t-ch) \geq 0,\] $ t < \sigma_*.$
Integrating the latter equation, we find that
\[
v(t) = C_1e^{\lambda t}+ C_2e^{\mu t}+
\frac{1}{\mu-\lambda}\int_0^t\left(e^{\mu(t-s)}-
e^{\lambda(t-s)}\right)D(s)ds,
\]
\[
\mbox{where} \quad C_1:= v(0)\frac{\mu -
v'(0)/v(0)}{\mu-\lambda}<0, \quad C_2:=
v(0)\frac{v'(0)/v(0)-\lambda}{\mu-\lambda}>0,
\]
and $0 < \lambda< \mu$ satisfy  $z^2-cz+\alpha_0 =0$. We note here
that a direct comparation of the latter equation with $z^2-cz+
\beta_\kappa e^{-zch} =0$ shows that $\lambda< \mu< \lambda_3$.
This also implies  that $c(t):= C_1e^{\lambda t}+ C_2e^{\mu t}> 0$
for $t >0$.
 Indeed, $c(t)$ is positive for sufficiently large $t$ and if $C_1e^{\lambda T}+ C_2e^{\mu T} =0$ for the rightmost $T$, then
\[
e^{(\mu-\lambda)T}= \frac{v'(0)-\mu v(0)}{v'(0)-\lambda
v(0)}\approx \frac{\lambda_3-\mu}{\lambda_3-\lambda} <1 \quad
\mbox{so that} \ T <0.\]
 All the above imply that 
$\sigma_* = +\infty$ and $v(t) > 0.5C_2e^{\mu t}$ for sufficiently
large $t$, contradicting to the second inequality of (\ref{2inw}). 
\end{proof}
\section{Global continuation of wavefronts}
This section contains the proof of Theorem \ref{ngz}.  It is divided into  three parts.
\subsection{Lyapunov-Schmidt reduction.}
Take a fixed  $(h_0,c_0) \in  \Int\,\mathcal{D}_{\frak L}$ and suppose
that there exists a monotone wavefront $u = \phi(\nu \cdot x+ct),
\ |\nu| =1$, for equation (\ref{pe}) considered with $h=h_0$, and
propagating at the velocity $c=c_0$. Then $\phi$ satisfies
(\ref{e1}) or, equivalently,  $(v,w)= (\phi(t), \phi'(t))$
is a solution of 
\begin{equation}\label{SEO}
 v'(t)=w(t), \quad  w'(t)=c
w(t)-f(v(t),v(t-r)).
\end{equation}
with $c=c_0,  r = c_0h_0=:r_0$.  In what follows,  the spaces
$C_\delta, C_\delta^1$ will be also considered with the fixed
parameters  $c=c_0,  h=h_0$. The change of  variables
$z_1+\phi(t)=v,  z_2+\phi'(t)=w$   transforms (\ref{SEO}) into \[
\frak{F}_{c_0}(z)=G(h,c,z),
 \]
 where we use the notation $z(t)=(z_1(t),z_2(t))$, $z_{jr}(t)=z_j(t-r)$,
\[G(h,c,z)= (0,  (c-c_0)\phi'+(c-c_0)z_2
                                                        +f(\phi,\phi_{r_0})
-f(z_1+\phi,z_{1r}+\phi_r)+a(\cdot)z_1+b(\cdot)z_{1r_0}).
\]
 By Theorem \ref{thIF},  there
exists a subspace $W\subset C^1_{\delta}$, $\codim(W)=1$, such
that $C_{\delta}^1=\ker(\frak{F}_{c_0})\bigoplus W$. Clearly,  the
restriction \[L:=\frak{F}_{c_0}\Big|_{W}:W\rightarrow C_{\delta}\]
is continuous one-to-one operator,  hence $L^{-1}$ exists and is
bounded.

Set $W_{\rho} : = W \cap \{z\in C^1_\delta: |z|_{1,\delta} <
\rho\}$. We have the following
\begin{lemma}\label{lip}
There exist $\rho_1, \rho_2, K >0$  such that
\begin{enumerate}
\item[(i)] $|G(h,c,z)-G(h,c,w)|_\delta \leq K |z-w|_\delta$  for
all $z, w \in \mathcal{U}_{\rho_1}(0)=\{z: |z|_\delta < \rho_1\}$
and $(h,c) \in \mathcal{U}_{\rho_2}(h_0,c_0)  = \{(h,c): |h-h_0|+
|c-c_0| < \rho_2\}$. \item[(ii)] $L^{-1}G(h,c,\cdot):W_{\rho_1}
\to W_{\rho_1}$ is well defined and is a contraction uniformly in
$(h,c) \in \mathcal{U}_{\rho_2}(h_0,c_0)$. \end{enumerate}
\end{lemma}
\begin{proof} $(i)$ Set $R(s):= (\phi+s w_1+(1-s)z_1,\phi_r+sw_{1r}+(1-s)z_{1r})$,  where  $z=(z_1,z_2),\ w=(w_1,w_2) \in C_{\delta}$. Then  there exists $s_0 \in (0,1)$ such that
\[
 |G(h,c,z)(t)-G(h,c,w)(t)|\leq |c-c_0||z_2-w_2| +
|f_1(R(s_0))-a(t)||w_1-z_1|+\]
\[
|f_2(R(s_0))-b(t)||w_{1r}-z_{1r}|= |c-c_0||z_2-w_2| +\] \[
|f_1(R(s_0))-f_1(\phi,\phi_{r_0})||w_1-z_1|+|f_2(R(s_0))-f_2(\phi,\phi_{r_0})||w_{1r}-z_{1r}|.
\]
Now, since $f_j(x,y), j =1,2,$ are continuous functions of real variables and $\phi(t)$
is bounded on $\R$, for each given $\sigma >0$ there exists $\rho_0>0$
such that $\sup_{t \in \R, s \in
[0,1]}|R(s)(t)-(\phi,\phi_{r_0})(t)| \leq \rho_0$ implies that
$|f_j(R(s))-f_j(\phi,\phi_{r_0})|< \sigma$. Since
\[|R(s)(t)-(\phi,\phi_{r_0})(t)| \leq  |\phi(t-r)-\phi(t-r_0)| +
|w_1(t)|+|z_1(t)|+|w_{1}(t-r)|+|z_{1}(t-r)|\leq \] \[ \sup_{s \in
\R}\phi'(s)|r-r_0| +2\sup_{s\in \R} |w_1(s)|+2\sup_{s\in \R}
|z_1(s)|\leq  |\phi'|_\delta |r-r_0| +4 \rho_1 < \rho_0\] for
sufficiently small $\rho_1,\rho_2$, we find that
\[
 |G(h,c,z)(t)-G(h,c,w)(t)|\leq
\sigma(|w(t)-z(t)|+|w_{1}(t-r)-z_{1}(t-r)|).
\]
Therefore, for all $z, w \in \mathcal{U}_{\rho_1}(0)$ and
$(h,c) \in \mathcal{U}_{\rho_2}(h_0,c_0)$, it holds that 
\[|G(h,c,z)-G(h,c,w)|_\delta \leq
\sigma(|w-z|_\delta+|w_{1}(\cdot -r)-z_{1}(\cdot-r)|_\delta)\leq 2\sigma\Theta
|w-z|_\delta.\] Here we use the continuity of the usual translation
operator $T_r:C_\delta \to C_\delta,  \ r =ch >0,$ defined by $T_rz(s)= z(s-r)$:
$\|T_r\| \leq \exp(-r\lambda_2(c_0))\leq
\exp(-(h_0+\rho_2)(c_0+\rho_2)\lambda_2(c_0))=:\Theta$.

$(ii)$ Take $\sigma< (2\Theta\|L^{-1}\|)^{-1}$ and observe that
$\lim_{r\to 0}|T_r\phi - \phi|_\delta=0$:
\[
|T_r\phi - \phi|_\delta \leq \sup_{s\leq
0}e^{-(\lambda-\delta)s}|\phi(s)-\phi(s-r)| + \sup_{s\geq
0}e^{-(\lambda_2+\delta)s}|\phi(s)-\phi(s-r)| \leq
\]
\[
r\left(\sup_{s\leq 0}e^{-(\lambda-\delta)s}\phi'(\theta(s))+
\sup_{s\geq 0}e^{-(\lambda_2+\delta)s}\phi'(\omega(s)) \right)
= O(r).
\]
Next, if $z\in W_{\rho_1}$ and  $(h,c) \in
\mathcal{U}_{\rho_2}(h_0,c_0)$, then
\[
|G(h,c,z)|_{\delta}=|G(h,c,z)-G(h,c,0)|_{\delta}+|G(h,c,0)|_{\delta}<
2\sigma \Theta\rho_1+|c-c_0||\phi'|_{\delta}+
\]
\[|f(\phi,\phi_{r_0})-f(\phi,\phi_r)|_{\delta}< 2\sigma
\Theta\rho_1+|c-c_0||\phi'|_{\delta}+{\displaystyle\max_{[0,\kappa]\times
[0,\kappa]}}|f_2(x,y)| |\phi_{r_0}-\phi_r|_{\delta} <
\frac{\rho_1}{\|L^{-1}\|},
\]
once $\rho_1,\rho_2, \sigma$ are sufficiently small. Therefore,
for the same $c, h, z,$ we have \[|L^{-1}G(h,c,z)|_{\delta,1}\leq
\|L^{-1}\||G(h,c,z)|_{\delta}< \rho_1,\] so that
$L^{-1}G(h,c,\cdot):W_{\rho_1} \to W_{\rho_1}$ is well defined.
Finally, for $h,c, z$ as above,
\[|L^{-1}G(h,c,z)-L^{-1}G(h,c,w)|_{\delta,1}\leq \|L^{-1}\|
|G(h,c,z)-G(h,c,w)|_{\delta} \leq 
2\sigma\Theta\|L^{-1}\||z-w|_{\delta}\] which completes the proof of the
lemma.
\end{proof}
\begin{corolary}\label{vt}Assume that  either hypothesis $({\bf MG})$ or  $({\bf KPP})$ holds. If $\phi(h_0,c_0)(t)$  is a monotone wavefront of  (\ref{e1}) for some
$(h_0,c_0) \in \Int\,\mathcal{D}_{\frak L}$  then  there exist $\rho>0$
and continuous map $\phi: \R^2_+\cap \mathcal{U}_{\rho}(h_0,c_0)
\to C^1_\delta$ such that each $\phi(h,c)(t)$ is a travelling front
of equation (\ref{e1}) considered with $(h,c) \in \R^2_+\cap
\mathcal{U}_{\rho}(h_0,c_0)$.
\end{corolary}
\begin{proof}
Indeed, since $L^{-1}G(h,c,\cdot):W_{\rho_1} \to W_{\rho_1}$  is a
uniform  contraction, there exist a unique solution $z=z(h,c)$  of
the equation $L^{-1}G(h,c,z)=z$.  Moreover, the function $z:
\mathcal{U}_{\rho_2}(h_0,c_0) \to W_{\rho_1}$  depends
continuously on $(h,c)$ (e.g. see \cite[Section 1.2.6]{Hen}) and
$z(h_0,c_0)=0$. As a consequence,
$\frak{F}_{c_0}(z(h,c))=G(h,c,z(h,c))$ and therefore
$\phi(h,c)(t):= \phi(t) + z_1(h,c)(t)$ is a travelling front of
equation (\ref{e1}) considered with $(h,c) \in \R^2_+\cap
\mathcal{U}_{\rho}(h_0,c_0)$.
\end{proof}
\subsection{Asymptotic analysis of $\phi(t,h,c):= \phi(h,c)(t)$.}
Fix  $(h_0,c_0) \in \mathcal{D}_{\frak N}$ and suppose
that there exists a monotone wavefront for equation (\ref{e1})
considered with $h=h_0$ and propagating with the velocity $c=c_0$. As we have
proved, this implies the existence of an open neighborhood
$\mathcal O \subset \R^2_+$ of  $(h_0,c_0)$  and a continuous
family $\phi: {\mathcal O}  \to C_\delta^1$ of wavefronts to
(\ref{pe}). It should be observed that, at the present moment, we do not have any information either about the positivity or about the monotonicity properties of $\phi(h,c)$. In the next lemma, we analyze the main term of
asymptotic expansions of  each particular wavefront $\phi(h,c)$
at  the infinity.  Recall that  $f\in C^{1,\gamma}$ for  some
$\gamma\in (0,1]$.  Since $\delta$ can be taken arbitrarily small,
there is no loss of generality in assuming that $\gamma, \delta$
satisfy $(\gamma+1)(\lambda(c_0)-\delta) > \lambda(c) + 2\sigma> \lambda(c_0) - \delta$
for some $\sigma >0$ and all $(h,c) \in \mathcal O$.
\begin{lemma}\label{afl} Let $(h_0,c_0) \in \mathcal{D}_{\frak N}$. Then there exist an open  neighborhood ${\mathcal O'} \subset \mathcal O$ and  continuous functions
$K_1,K_2:  \mathcal O' \to (0,+\infty)$ such that, for some
$\sigma >0, M >0,$ independent of $c,h$, and for all $(h,c) \in
\mathcal O'$, it holds that 
\[\hspace{-7mm}(\phi (t,h,c),\phi' (t,h,c))=
\left\{
\begin{array}{cc}
K_1(h,c) e^{\lambda(c)
t}(1, \lambda(c)) + R_1(t,h,c),  &  t \leq 0,  \\
(\kappa,0) - K_2(h,c)e^{\lambda_2(c) t}(1, \lambda_2(c))
+R_2(t,h,c),   & t\geq 0,
\end{array}
\right.
\]
where $ |R_1(t,h,c)| \leq M e^{(\lambda(c)+\sigma) t}, \ t \leq
0,\quad |R_2(t,h,c)| \leq M e^{(\lambda_2(c)- \sigma) t}, \ t \geq
0. $
\end{lemma}
\begin{proof} First, we will analyze the asymptotic behavior at $-\infty$. By Corollary \ref{vt},   there exist a positive number $M_1>0$ and an open  neighborhood ${\mathcal O_1} \subset \mathcal O$ such that, for all  $t \leq 0, \ (h,c) \in {\mathcal O_1}$,
\[|\phi(t,h,c)| = |z(t,h,c)+\phi(t)| \leq (|z(h,c,\cdot)|_\delta +
|\phi|_\delta)e^{(\lambda(c_0)-\delta)t}\leq
M_1e^{(\lambda(c_0)-\delta)t}.\] Since
\begin{equation}\label{apf}
\phi''(t,h,c)-c\phi'(t,h,c)+\alpha_0
\phi(t,h,c)+\beta_0\phi(t-ch,h,c)=F(t,h,c),
\end{equation}
where $F(t,h,c):=\alpha_0
\phi(t,h,c)+\beta_0\phi(t-ch,h,c)-f(\phi(t,h,c),\phi(t-ch,h,c))$
satisfies \[ |F(t,h,c)|
 \leq 
|\alpha_0-f_1(\theta(t)\phi(t),\theta(t)\phi(t-ch)||y(t)|+|\beta_0-f_2(\theta(t)\phi(t),\theta(t)\phi(t-ch))||\phi(t-ch)|
\]
\[ \leq C_1(|\theta(t)\phi(t)|+|\theta(t)\phi(t-ch)|)^{\gamma}(|\phi(t)|+|\phi(t-ch)|)\leq
C_2e^{(\gamma+1)(\lambda(c_0)-\delta)t},\] with $C_j$ independent
of $(h,c)\in \mathcal{O}_1$ and $\theta(t)\in (0,1)$ appearing due
to an application of  the mean value theorem. Thus $|F(t,h,c)|
\leq C_2e^{(\lambda(c)+2\sigma)t}, \ t \leq 0,$ so that, by
\cite[Lemma 28]{GT}, $ \phi(t,h,c)=w_{-}(t)+u_-(t)$,
where
\[w_{-}(t)=-\mbox{Res}_{z=\lambda(c)}\left(\frac{e^{zt}}{\chi_0(z)}\int_{\R}e^{-zs}F(s,h,c)ds\right)=-e^{\lambda(c)t}\frac{\tilde
F(\lambda(c),h,c)}{\chi_0'(\lambda(c))},\]
\[
u_-(t) =\frac{
e^{(\lambda(c)+\sigma)t}}{2\pi}\int_{\R}e^{ist}\frac{\tilde{F}(\lambda(c)+\sigma
+is,h,c)}{\chi_0(\lambda(c)+\sigma+is,h,c)}ds, \ \tilde F(z,h,c):
= \int_{\R}e^{-zs}F(s,h,c)ds,
\]
whenever $\delta,\sigma$ are sufficiently small positive numbers.
Set
\[
K_1(h,c):=-\frac{\int_{\R}e^{-\lambda(c)s}F(s,h,c)ds}{\chi_0'(\lambda(c))}.
\]
Since continuous  $F(t,h,c)$ is uniformly bounded on $\R\times
\mathcal{O}_1$ and, in addition, $e^{-\lambda(c)s}|F(s,h,c)|\leq
C_2e^{2\sigma s}, \ s \leq 0,$ we conclude that
 $K_1(h,c)$ is also continuous on $\mathcal{O}_1$. We note that
$\chi_0'(\lambda(c))<0$ for all $(h,c)\in \mathcal{O}_1 \subset 
\Int\, \mathcal{D}_{\frak L}$. Next,  there exists
an open subset  $\mathcal{O}_2\subset\mathcal{O}_1$ such that, for $(h,c)\in \mathcal{O}_2$,
\[
|u_-(t,h,c)|\leq \frac{e^{(\lambda(c)+\sigma)t}}{2\pi}\
\int_{\R}\frac{1}{|\chi_0(\lambda(c)+\sigma+is,h,c)|}\int_{\R}e^{-t(\lambda(c)+\sigma)}|F(t,h,c)|dtds
\leq e^{(\lambda(c)+\sigma)t} C_3, \] where $C_3$ is independent
of $(h,c)$. Indeed, as we have seen, the function
$\int_{\R}e^{-t(\lambda(c)+\sigma)}|F(t,h,c)|dt$ is uniformly
bounded on $\mathcal{O}_1$ and, on the other hand,  for some open subset  $\mathcal{O}_2\subset\mathcal{O}_1$ and positive
$C_4,C_5,$ it holds  that $C_4+ C_5s^2 \leq
|\chi_0(\lambda(c)+\sigma+is,h,c)|, \ s \in \R,\ (h,c)\in
 \mathcal{O}_2$.

In consequence, $K_1(h_0,c_0)\not=0$, since otherwise
$\Lambda_-(\phi) \geq \lambda(c_0)+\sigma > \lambda(c_0)$ (recall
that $(h_0,c_0) \in \mathcal{D}_{\frak N}$ and see Definition
\ref{lpm}).  Moreover, the positivity of $\phi$ implies that
$K_1(h_0,c_0)>0$.  Since $K_1(h,c)$ is continuous,  there exists an
open set $\mathcal{O}_3\subset\mathcal{O}_2$ where $K_1(h,c)$ is
positive.

Next, after integrating equation (\ref{apf}) on $(-\infty,t)$, we
obtain
\begin{eqnarray*}
\hspace{-10mm}\phi'(t,h,c)=c\phi(t,h,c) +
\int_{-\infty}^t\left(F(s,h,c)-\alpha_0
\phi(s,h,c)-\beta_0\phi(s-ch,h,c)\right)ds=\\
K_1(h,c)\lambda(c)e^{\lambda(c)t}+cu_-(t)
+\int_{-\infty}^t\left(F(s,h,c)-\alpha_0
u_-(s)-\beta_0u_-(s-ch)\right)ds,
\end{eqnarray*}
that proves the asymptotic formula of Lemma \ref{afl} at
$-\infty$.

After applying the change of variables
$y(t,h,c)=\kappa-\phi(t,h,c)$, the study of the asymptotic
behavior of wavefronts at $+\infty$ becomes fully analogous
to the first case and is left to the reader.
\end{proof}

\vspace{-4mm}

\subsection{The final part of the proof of Theorem \ref{ngz}.}
The proof of our main result is an easy consequence of the
following three propositions.

\vspace{-4mm}

\begin{lemma}\label{vt3} Assume that  either hypotheses $({\bf M})\& ({\bf MG})$ or  $({\bf M})\&({\bf KPP})$ are satisfied and $(h_0,c_0) \in \mathcal{D}_{\frak N}$. Then  in Corollary \ref{vt},  we can choose $\rho>0$ such that $\phi(h,c)(t)$ is a positive monotone wavefront of
(\ref{e1}) for each  $(h,c) \in \R^2_+\cap
\mathcal{U}_{\rho}(h_0,c_0)$. Hence,  the non-empty set
\[\mathcal{D}'_{{\frak N}}:= \{(h,c) \in \mathcal{D}_{{\frak N}}:
\mbox{there is at least one monotone wavefront for  (\ref{e1})}
\}\] is open in topology of $\mathcal{D}_{{\frak N}}$.
\end{lemma}
\begin{proof} First, we observe that $\{0\}\times (c^{\frak N}_*,+\infty)\subset \mathcal{D}'_{{\frak N}}\not=\emptyset$ because of the existence results and asymptotic formulae (\ref{afe}) presented in the second paragraph of the introduction.  Next, 
by Lemma \ref{afl}, there exist $\rho'>0$ and $T>0$ independent of
$h,c,$ such that  $\phi'(h,c)(t),\ \phi(h,c)(t) >0$ for all $|t|
\geq T, \  (h,c) \in \mathcal{U}_{\rho'}(h_0,c_0)\cap \R^2_+$.  On
the other hand,  due to the continuity of application $\phi:
\R^2_+\cap \mathcal{U}_{\rho'}(h_0,c_0) \to C^1_\delta$, we find
that, for an appropriate $\epsilon >0$ and some $0< \rho < \rho'$,
it holds that 
\[0<\phi(h_0,c_0)(t)-\epsilon<\phi(h,c)(t)<\epsilon+\phi(h_0,c_0)(t)<\kappa,
\   |t|\leq T, \ (h,c) \in \mathcal{U}_{\rho}(h_0,c_0). \] In
consequence, $ \phi(h,c)(t) \in (0,\kappa)$ for all $t \in \R, \
(h,c) \in \R^2_+\cap \mathcal{U}_{\rho}(h_0,c_0)$.  In addition,
by assumption ({\bf M}),   each profile $\phi(h,c)(\cdot): \R \to
(0,\kappa)$  is a monotone function.

Finally, it is clear that  $\Lambda_-(\phi(h,c)) = \lambda(c),
\Lambda_+(\phi(h,c)) = \lambda_2(c)$ for each
 $(h,c) \in \mathcal{U}_{\rho}(h_0,c_0)$. This means that  $\R^2_+\cap \mathcal{U}_{\rho}(h_0,c_0) \subset  \mathcal{D}'_{\frak N}$.  Since  $(h_0,c_0)$ was an arbitrary point from  $\mathcal{D}'_{\frak N}$, we conclude that $\mathcal{D}'_{\frak N}$ is open in $\mathcal{D}_{\frak N}$.
\end{proof}
\begin{lemma}\label{vt4} For each $(h_0,c_0) \in \overline{
\mathcal{D}'_{\frak N}}$, equation  (\ref{e1}) has at least one
positive monotone front. Therefore $\mathcal{D}'_{{\frak N}}$ is
closed  in topology of $\mathcal{D}_{{\frak N}}$ so that
$\mathcal{D}'_{{\frak N}}=\mathcal{D}_{{\frak N}}$.
\end{lemma}
\begin{proof} Suppose that a sequence of points $(h_n,c_n)\in \mathcal{D}'_{{\frak N}}$ converges to $(h_0,c_0)$.  If we denote by $\phi_n(t)$ some associated sequence of monotone wavefronts normalized by $\phi_n(0)=\kappa/2$,  a direct verification shows that
\begin{equation}\label{psea} \hspace{-7mm}
\phi_n(t) = \frac{1}{z_2-z_1}\left\{\int_{-\infty}^te^{z_1
(t-s)}(\mathcal{H}\phi_n)(s)ds + \int_t^{+\infty}e^{z_2
(t-s)}(\mathcal{H}\phi_n)(s)ds \right\},
\end{equation}
where $(\mathcal{H}\phi)(s)= \phi(s)+ f(\phi(t),\phi(t-ch))$ and
$z_1<0<z_2$ satisfy $z^2 -cz -1 =0$.  It follows from (\ref{psea})
that $0\leq \phi_n'(t)\leq \kappa + \max_{[0,\kappa]^2}|f(x,y)|$.
Thus $\{\phi_n(t)\}$ has a subsequence (by abusing the notation,
we will call it again  $\{\phi_n(t)\}$) converging in the compact
open topology of $C(\R,\R)$.  Let $\phi_0= \lim \phi_n$, passing
to the limit (as $n \to +\infty$) in (\ref{psea}), we find that
$\phi_0(t)$ also satisfies (\ref{psea}). Therefore $\phi_0(t), \
\phi_0(0)=\kappa/2, \ \phi_0(t) \leq \kappa, 0\leq \phi_0'(t)\leq
\kappa + \max_{[0,\kappa]^2}|f(x,y)|,$ is a monotone  positive
solution of   (\ref{e1}). Since $\phi_0(\pm\infty)$ are finite and
$\phi''_0(t)$ is bounded, we obtain that  $\phi_0'(\pm\infty)=0$.
In consequence, taking into account that $\phi_0$ is a bounded
solution of equation (\ref{e1}), we find that
$f(\phi_0(\pm\infty),\phi_0(\pm\infty)) =0$.  In this way,
$\phi_0(-\infty) =0,  \phi_0(+\infty) =\kappa$. Since
$\beta_\kappa <0$,  it follows from (\ref{e1}) that actually $0<
\phi_0(t) < \kappa, \ t \in \R$.  Finally, since
$\mathcal{D}'_{{\frak N}}$ is simultaneously closed and open in connected space $\mathcal{D}_{{\frak N}}$,  we
obtain that $\mathcal{D}'_{{\frak N}}=\mathcal{D}_{{\frak N}}$.
\end{proof}
\begin{lemma}\label{vt4} If $(\bar h,\bar c) \in \R^2_+\setminus \mathcal{D}_{{\frak L}}$, then equation  (\ref{e1}) does not have any positive eventually monotone front.
\end{lemma}
\begin{proof} Take some $(\bar h,\bar c) \in \R^2_+\setminus \mathcal{D}_{{\frak L}}$. Then either  $\bar c< c^{\frak L}_0(\bar h)$ or  $\bar c > c^{\frak L} _\kappa(\bar h)$. In the first case, the non-existence of positive fronts is a well known fact (cf. \cite[Theorem 1]{TAT}).  Consequently, it suffices  to consider the case  $\bar c > c^{\frak L} _\kappa(\bar h)$. Then Lemma \ref{lc2} implies that
$\chi_\kappa(z)$ does not have negative zeros. Arguing by
contradiction, suppose that, nevertheless, equation (\ref{e1}) has some positive
eventually monotone front $\phi(t)$ for  $h=\bar h,c=\bar c$. Then
$\psi(t):= \pm(\kappa-\phi(t))$ is strictly positive on some
interval $[T,+\infty)$ and satisfies
\[
\psi''(t)-\bar c\psi'(t) \pm f(\kappa \pm \psi(t), \kappa \pm
\psi(t-\bar c\bar h))=0,  \ \psi(+\infty)=0,
\]
where the sign "$-$" [respectively, "$+$"] corresponds to the case
$\phi(t) < \kappa, \ t > T$ [to the case  $\phi(t) > \kappa, \ t >
T$,  respectively]. Following the approach in \cite{GT}, we will
show that the inequality  $\bar c > c^{\frak L} _\kappa(\bar h)$
will force $\psi(t)$ to oscillate about the zero.
For the convenience of the reader, the proof is divided in several
steps.

{\it Claim I:  $\psi(t)$ has at least  exponential decay as $t \to
+\infty$.}

First, observe that
\begin{equation}\label{ph}  \psi''(t) - \bar c\psi'(t) = \Gamma \psi(t) - g(t), \
t \in \R,
\end{equation}
where, with some $\mathbf{z}(t):= (\kappa \pm \theta(t)\psi(t),
\kappa \pm \theta(t)\psi(t-\bar c\bar h))$,  $\theta(t) \in
(0,1), \ \Gamma >0,$ we set  \[g(t):= \Gamma \psi(t)\pm f(\kappa \pm \psi(t),
\kappa \pm \psi(t-\bar c\bar h)) =(\Gamma +
f_1(\mathbf{z}(t)))\psi(t) +  f_2(\mathbf{z}(t))\psi(t- \bar c\bar
h).\] Since $f_1(\mathbf{z}(+\infty))  + f_2(\mathbf{z}(+\infty))
=\alpha_\kappa+ \beta_\kappa <0$, $f_2(\mathbf{z}(+\infty)) =
\beta_\kappa <0,$ and $\psi(t)$ is decreasing, we find that, for
all sufficiently large $t$ and some positive $0< \Gamma <
-\beta_\kappa-\alpha_\kappa$, it holds that 
\[
g(t) \leq (\Gamma + f_1(\mathbf{z}(t)) +
f_2(\mathbf{z}(t)))\psi(t)<0.
\]
Since $\psi(t), g(t)$ are bounded on $\R$, we obtain that
\[
\psi(t)= \frac{1}{m-l}\left(\ \int_{-\infty}^te^{l(t-s)}g(s)ds +
\int_t^{+\infty}e^{m(t-s)}g(s)ds\right),
\]
where $l<0$ and $0<m$ are roots of  $z^2 - \bar cz - \Gamma =0$.
The latter representation of $\psi(t)$ implies that there exists
$T_0$ such that
\begin{equation}\label{yyp}
\psi'(t)- l\psi(t) = \int_t^{+\infty}e^{m(t-s)}g(s)ds < 0, \ t
\geq T_0.
\end{equation}
Hence, $(\psi(t)\exp(-lt))' < 0,\  t \geq T_0,$ and therefore
\begin{equation}\label{yg}
\psi(t) \leq \psi(s)e^{l(t-s)}, \quad t\geq s \geq T_0, \quad
g(t)= O(e^{lt}), \ t \to +\infty.
\end{equation}
 Finally, (\ref{yyp}), (\ref{yg}) imply that
$\psi'(t)= O(e^{lt}), \ t \to +\infty$.

{\it Claim II: $\psi(t)>0$ is not
superexponentially small as $t \to +\infty$.}\\
Recall  that $\psi(t)$ is decreasing and positive on  $\R$. Since
the right hand side of Eq. (\ref{ph}) is positive and integrable
on $[T_0, +\infty)$, and since $\psi(t)$ is a bounded solution of
(\ref{ph}) satisfying $\psi(+\infty)=0$, we find that
\[
\psi(t) = -\int_t^{+\infty}(1-e^{\bar c(t-s)})(
f_1(\mathbf{z}(s))\psi(s) +  f_2(\mathbf{z}(s))\psi(s- \bar c\bar
h))ds.
\]
As a consequence, there exists $T_1$ such that \[ \psi(t) \geq
0.5|\beta_\kappa|(1-e^{-0.5\bar h\bar c}) \int_{t-0.5\bar h\bar
c}^{t}\psi(s)ds:= \xi \int_{t-0.5\bar h\bar c}^{t}\psi(s)ds, \quad
t \geq T_1-\bar c \bar h.
\]
Now, since $\psi(t) >0$ for all $t$, we can find positive $C,
\rho$ such that $\psi(s) > C e^{-\rho s}$ for all $s \in [T_1-\bar
c \bar h, T_1]$. We can assume that $\rho$ is large enough to
satisfy the inequality $ \xi(e^{0.5\rho \bar h\bar c}-1)> {\rho}$.
Then we claim that $\psi(s) > C e^{-\rho s}$ for all $s \geq
T_1-\bar c\bar h$. Conversely, suppose that $t'
> T_1$ is the leftmost point where $\psi(t') = C e^{-\rho t'}$. Then we get a contradiction:
\[
\psi(t') \geq \xi \int_{t'-0.5\bar h\bar c}^{t'}\psi(s)ds > C\xi
\int_{t'-0.5\bar h\bar c}^{t'}e^{-\rho s}ds = Ce^{-\rho
t'}\xi\frac{e^{0.5\rho \bar c \bar h}-1}{\rho} > Ce^{-\rho t'}.\]
 {\it Claim III: $\psi(t)> 0$ can not hold when $\chi_\kappa(z)$ does not have
any zero in $(-\infty,0)$}. 

Observe that $\psi(t)$ satisfies
\[\psi ''(t) - \bar c\psi'(t)+ f_1(\mathbf{z}(t))\psi(t) +
f_2(\mathbf{z}(t))\psi(t- \bar c\bar h)=0,\ t \in \R,\
\]
where in virtue of Claim I, it holds that $(\psi(t),\psi'(t))= O
(e^{lt})$. Next, $f \in C^{1,\gamma}$ assures that
$f_1(\mathbf{z}(t)) = \alpha_k + O(\psi^\gamma(t))$,
$f_2(\mathbf{z}(t)) = \beta_k + O(\psi^\gamma(t))$ at $t =
+\infty$.  Then \cite[Proposition 7.2]{FA}  implies that there
exists $q <l$ such that $\psi(t) = v(t) + O(e^{q t}),$ $ t
\to +\infty,$ where $v$ is a {\it non empty} (due to Claim II)
finite sum of eigensolutions of the limiting equation
\[
y''(t) - \bar cy'(t)+ \alpha_ky(t)+ \beta_\kappa y(t-\bar c \bar h)=0,\ t \in
\R,\
\]
associated to the eigenvalues $\lambda_j \in F= \{q < \Re\lambda_j
\leq l\}$. Now, since the set $F$ does not contain any real
eigenvalue by our assumption, we conclude that $\psi(t)$ should be
oscillating on $\mathbb{R}_+$,  a contradiction. \end{proof}
\section{Appendix}
\subsection{Proof of Lemma \ref{lc2}} \label{s5.1}With $\lambda:=cz, \  \epsilon = c^{-2},$ equation (\ref{char1a}) takes the form
\begin{equation}\label{char1b}
F(\lambda):= \epsilon \lambda^2 - \lambda + \alpha_\kappa +\beta_\kappa e^{-h\lambda}
=0,  \ \epsilon >0.
\end{equation}
Since $F'''(x) >0, \ x \in \R$, equation (\ref{char1b}) has at
most three real roots.  Since $F(0) <0, \  F(\pm\infty) =\pm\infty,$ this equation has an even number (either 0 or 2) of negative roots (counting the multiplicity) and
at least one positive root. A straightforward analysis of
(\ref{char1b}) shows  that

(a) If this equation has a negative root for some $\epsilon_0 \geq
0$, it also has two negative roots for each $\epsilon
>\epsilon_0$. We will denote the greatest negative root as
$\lambda_2$.  If $\epsilon_0 =0$, we obtain $c^{\frak L}_\kappa(h)= +\infty$.

(b) If equation (\ref{char1b}) does not have any negative root for
$\epsilon =0$ (this happens when $h$ is sufficiently large), there
exists a unique $\epsilon_0 >0$ such that (\ref{char1b}) possesses
two negative roots (counting the multiplicity)  for $\epsilon \geq
\epsilon_0$ and does not have a negative root for $\epsilon <
\epsilon_0$. Thus $c^{\frak L}_\kappa(h) = \epsilon_0^{-1/2}$ is finite for
sufficiently large $h$ and $\epsilon_0 = \epsilon_0(h)$ can be
determined from the system
\begin{equation}\label{char1c}
\epsilon \lambda^2 - \lambda  +\alpha_\kappa =  -\beta_\kappa e^{-h\lambda}, \ 2\epsilon
\lambda - 1 =  h \beta_\kappa e^{-h\lambda}.
\end{equation}
In particular,  the double negative root $\lambda = \lambda(h)$ of
(\ref{char1b}) satisfies
\begin{equation}\label{con}
-2\frac{\alpha_\kappa}{ \beta_\kappa} + \frac{\omega}{ \beta_\kappa h}= e^{-\omega}(2+\omega), \quad
\omega: = h\lambda(h),
\end{equation}
while $c^{\frak L}_\kappa(h)$ is strictly decreasing on some maximal open interval
$(h_0, +\infty), \ h_0 >0, $ because of
$\epsilon_0'(h) = \beta_\kappa e^{-\lambda h}/\lambda >0.
$
Observe that $\alpha_\kappa/ |\beta_\kappa| <1$ and the right-hand side of (\ref{con}) has a unique inflection 
point at  $\omega =0$. This implies that $\omega(h) \to
\omega_\kappa, \ h \to +\infty,$ where $\omega_\kappa<0$ satisfies
(\ref{roe}).

It is clear that  $c^{\frak L}_\kappa(h)=+\infty$ for $h \in [0, h_0]$.  From the
second equation of (\ref{char1c}), we also easily obtain that
$
\lim_{h \to +\infty} hc^{\frak L}_\kappa(h)=
\sqrt{\frac{2\omega_\kappa}{ \beta_\kappa}}e^{\omega_\kappa/2},
$
so that $c^{\frak L}_\kappa(+\infty)=0$.

(c) It is immediate to see that, for each fixed $c = 1/\sqrt{\epsilon} \in (0,c^{\frak L}_\kappa]$,  there
exists $x_1>0$ (independent on $h$) such that $\Re \lambda_j <
x_1$ for every $\lambda_j$ satisfying (\ref{char1b}). Furthermore,
for every fixed $x_2 \in \R$ there is an increasing continuous
function  $y=y(h) >0,\ h \geq 0,$ such that all roots $\lambda_j$
of (\ref{char1b}) with $\Re \lambda_j \geq x_2$ are contained in
the rectangle ${\mathcal R}(x_2,h):=[x_2,x_1]\times [- y(h), y(h)]
\subset \C$.  Next, observe that because of $\alpha_\kappa +\beta_\kappa <0$ equation
(\ref{char1b}) with $h =0$ has only two roots $\lambda_2 < 0 <
\lambda_1$.  By the Rouche's theorem, this implies that, for all
small positive $h$, equation   (\ref{char1b})   does not have roots
$\lambda_j= \lambda_j(h),\  \Re \lambda_j \geq \lambda_2,$ others
than $\lambda_2(h), \lambda_1(h)$. Now, suppose for a moment that
for some positive $h_0$ there exists complex $\lambda_j(h_0) \in
{\mathcal R}(\lambda_2(h_0),h_0)$. Let  $h_0$ be the minimal value with
such a property, then  the Rouche's theorem assures that
$\Re\lambda_j(h_0) = \lambda_2(h_0)$. Moreover, $\Im\lambda_j(h_0)
\not=0$ since otherwise $\lambda_2(h_0)$ would have the multiplicity 3.
Thus equation (\ref{char1b})  with $h=h_0$ has at least three
roots of the form $\lambda(y) := \lambda_2(h) + i y$ with $y \in
\{-\theta,0, \theta\}$ for some positive $\theta$.    Since  $c
\in (0,c^{\frak L}_\kappa]$, the function  $F(x)$ has exactly two critical
points, one of them belongs to $[\lambda_3(h),\lambda_2(h)]$ and
the second one is in $(\lambda_2(h),\lambda_1(h))$. In
consequence, \[F'(\lambda_2(h)) = 2\epsilon \lambda_2(h) - 1 +h
|\beta_\kappa|e^{-h\lambda_2(h)} \leq 0.\]  However, this contradicts to the
following relations: $F(\lambda(\theta)) = 0= \Im
F(\lambda(\theta))= $ \[\theta \left(2\epsilon  \lambda_2(h) - 1 +
h |\beta_\kappa|e^{-h\lambda_2(h)}\frac{\sin(h\theta)}{h\theta}\right) <
\theta \left(2\epsilon  \lambda_2(h) - 1 + h
|\beta_\kappa|e^{-h\lambda_2(h)}\right) \leq 0.  \hspace{1.8cm} \square \]
\subsection{Proof of Lemma \ref{lc1}} The existence of the critical speed $c^{\frak L}_0(h)$ which  has properties 
mentioned in the lemma is 
a well known fact, and its proof is omitted.  Clearly, it suffices to consider the case $\beta_0 >0$. Next, if  $c > c^{\frak L}_0(h)$ then 
$0< Q_0 := cq_0 -q_0^2 - \alpha_0 < \beta_0 e^{-ch q_0}$ for some  $q_0 = q_0(c) \in (0, \lambda),$ $ \ c - 2q_0 >0$. 
The change of variables $ \omega: = (z-q_0)(c -2q_0)/Q_0 $
transforms  (\ref{char1}) into
\begin{equation}\label{char1om}
\epsilon \omega^2 - \omega -1 + \gamma e^{- \omega h'}=0,
\end{equation}
where \[ \epsilon: = \frac{Q_0}{(c-2q_0)^2} >0, \ \gamma := \beta_0
e^{-ch q_0}/Q_0 >1, \ h' := \frac{chQ_0}{c-2q_0} >0.
\]
Now, since  inequalities (\ref{por1}) for equation
(\ref{char1om}) were established  in \cite[Lemma 2.3]{TT}, we
obtain that inequalities (\ref{por1}) hold also for equation
(\ref{char1}) once $c > c^{\frak L}_0(h)$.

Next, let $z_0 = x_0 + i y_0$ with $\Re z_0 = x_0 <  \lambda$ be a
complex root of (\ref{char1}). Then $ 0> (2x_0-c)|y_0| = \beta_0
e^{-chx_0} \sin (ch|y_0|) $ and therefore $ch|y_0| > \pi$.

Finally, the derivation of asymptotic representation  and the proof of monotonicity of $c^{\frak L}_0(h)$ repeat the 
arguments used in Subsection \ref{s5.1} (b) above and are omitted. \hfill $\square$

\subsection{Proof of Lemma \ref{lc3}}
Suppose that the graphs of the functions $c= c^{\frak L}_0(h)$ and $c= c^{\frak L}_\kappa(h)$ intersects
at some $h=h_1$. Since $\lambda_2(h_1) < 0 < \lambda(h_1)$, after differentiating the first equation of (\ref{char1c}) with respect to $h$, 
we obtain 
\[\frac{d}{dh}c^{\frak L}_\kappa(h)
|_{h=h_1} = -\frac{c^{\frak L}_\kappa(h_1)}{h_1} +
\frac{(c^{\frak L}_\kappa(h))^3}{2h_1\lambda_2(h_1)} <
-\frac{c^{\frak L}_0(h_1)}{h_1} +
\frac{(c^{\frak L}_0(h))^3}{2h_1\lambda(h_1)}=
\frac{d}{dh}c^{\frak L}_0(h)
|_{h=h_1}. \hspace{1.8cm} 
\]
This means that the above mentioned graphs have a unique transversal intersection on $\R_+$. 
As a consequence, if  $\theta(\alpha_\kappa,\beta_\kappa) = \theta_1(\alpha_0,\beta_0)$ then 
$c^{\frak L}_0(h) <  c^{\frak L}_\kappa(h)$ for all $h \geq 0$. \hfill $\square$
 
\subsection{Proof of Lemma \ref{tdo}}
It suffices to prove the inclusion $\Int\,\mathcal{D}_{\frak{L}} \subset \mathcal{D}_{\frak{N}}$ where  $\Int \, \mathcal{D}_{\frak{L}}$ denotes  the interior of $\mathcal{D}_{\frak{L}}$.  So let us fix some
$(h,c)\in \Int\,\mathcal{D}_{\frak{L}}$. By the definition of $\mathcal{D}_{\frak{N}}$,  it holds automatically 
$(h,c)\in \mathcal{D}_{\frak{N}}$ if there does not exist any monotone 
heteroclinic solution to equation (\ref{e1}) for the choosen pair $(h,c)$. Therefore we can assume that 
(\ref{e1}) has a positive monotone front 
$\phi: \R \to (0,\kappa)$.  Set $y(t):=\kappa-\phi(t)$
and $\mathbf{u}(t)= (\kappa - sy(t), \kappa - sy(t-ch)), \ s \in [0,1]$. Then
$y(t)$ satisfies the linear equation
 \begin{equation}\label{kppOaux}
 x''(t)-cx'(t)+(\alpha_\kappa+N(t))x(t)+ (\beta_\kappa+M(t))x(t-ch)=0,
 \end{equation}
 \[ \mbox{where} \hspace{0.5cm} N(t):=\int_0^1f_1(\mathbf{u}(t))ds - \alpha_\kappa, \
M(t):=\int_0^1f_2(\mathbf{u}(t))ds - \beta_\kappa,\
\]
so that $N(+\infty)= M(+\infty) =0$.   Since the linear equation with
constant coefficients
\begin{equation}\label{lab}
 x''(t)-cx'(t)+\alpha_\kappa x(t)+ \beta_\kappa x(t-ch)=0
\end{equation}
is hyperbolic (i.e. it does not have eigenvalues on the imaginary axis) 
and $N(+\infty)= M(+\infty) =0$,   equation  (\ref{kppOaux})  possesses  the
property of  exponential dichotomy 
on some infinite interval $[\tau,+\infty)$ (e.g. see
\cite[Lemma 4.3]{HL}). In particular,
$y(+\infty)=y'(+\infty)=0$ yield  $y(t), y'(t)=O(e^{-\rho t}), \ t \to +\infty,$ for
some $\rho>0$.  Therefore, in view of $C^{1,\gamma}$-smoothness of
$f$,  we have that $M(t),N(t)=O(e^{-\rho\gamma t})$ at  $t=
+\infty$.  Hence, invoking  \cite[Proposition 7.2]{FA} and Lemma
\ref{lc2}, we obtain that $y(t) = ae^{\lambda_2 t} +
o(e^{(\lambda_2-\delta) t}), \ t \to +\infty,$ for some $a$ and $
\delta >0$.  Note that $a\geq 0$ since  we have
$\phi(t) \in (0,\kappa)$ for  all $t \in \R$. 
In fact, $a$ can be  found explicitly (e.g., see
\cite[Lemma 28]{GT}): 
\begin{equation}\label{57}
a=
\mbox{Res}_{z=\lambda_2}\frac{-1}{\chi_\kappa(z)}\int_{\R}e^{-zs}S(s)ds=
\frac{-1}{\chi_\kappa'(\lambda_2)}\int_{\R}e^{-\lambda_2s}S(s)ds
>0,
\end{equation}
since  $$S(t):= N(t)y(t)+
M(t)y(t-ch) =
\alpha_\kappa(\phi(t)-\kappa)+\beta_\kappa(\phi(t-ch)-\kappa)-
f(\phi(t), \phi(t-ch)) \geq 0,$$ is not identically zero. Indeed,  if $S(t) \equiv 0$ then  bounded and strictly
decreasing $y(t)$ must satisfy  (\ref{lab}). However, this is impossible due to the hyperbolicity of this
equation. 
Thus  $\Lambda_+(\phi(t))=\lambda_2$. The proof of the relation
$\Lambda_{-}(\phi(t))=\lambda$ is completely similar and is left
to the reader. \hfill $\square$
\begin{remark}\label{ru} The above argument needs a minor modification to imply the inclusion
$\Int\, {\frak{D}}^-_{\frak{L}}:= \{(h,c): c \in (c_0^{\frak{L}}(h), c^-_\kappa(h)), \ h  \in [0,h^-_0]\} \subset \overline{\frak{D}}_{\frak{N}}$ stated in the proof of  Theorem \ref{TR}.  It suffices to show that $a$ in (\ref{57})
is positive  for each $(h,c) \in \Int\, {\frak{D}}^-_{\frak{L}}$. Assuming,  on the contrary, that $a=0$,  and again  invoking  \cite[Proposition 7.2]{FA} and Lemma
\ref{lc2}, we find that $y(t) = be^{\lambda_1 t} +
o(e^{(\lambda_1-\delta) t}), \ t \to +\infty,$ for some $b \geq 0, \
\delta >0$. Then $y(t)$ satisfies the equation 
$$
 x''(t)-cx'(t)+\alpha_\kappa x(t)+ \beta_\kappa^-x(t-ch)=Q(t),
$$
where $Q(t) := \alpha_\kappa y(t)+ \beta_\kappa^-y(t-ch) - 
f(\phi(t), \phi(t-ch)) \geq 0$. Furthermore, $Q(t)=  -N(t)y(t)+ (\beta_\kappa^--\beta_\kappa-M(t))y(t-ch)=O(e^{\lambda_1 t}), \ t \to +\infty$, and we claim that $Q(t)$  is not identically zero.  Indeed, 
by the proof of Lemma \ref{lc2}, we have that for $c \in (c_0^{\frak{L}}(h), c^-_\kappa(h))$  the characteristic function $\chi_\kappa^-(z):= z^2 - cz + \alpha_\kappa + \beta_\kappa^-
e^{-chz}$ has exactly three real zeros $\lambda_1^- < \lambda_2^- <  0 < \lambda_3^-$ and does not have any zero  on $i\R$. Moreover,  it is easy to see that  
$\lambda_1 \leq  \lambda_1^-< \lambda_2^- \leq \lambda_2 < 0$. Therefore,  if $Q(t) \equiv 0$ then  bounded and strictly decreasing $y(t)$ must satisfy  a hyperbolic equation with constant coefficient, a contradiction. Since 
$y(t) = O(e^{\lambda_2^- t}), \ t \to +\infty$, we also have that $y(t) = ce^{\lambda_2^- t} +
o(e^{(\lambda_2^--\delta_1) t}), \ t \to +\infty,$ for some $c \geq 0, \
\delta_1 >0$, where actually
$$
c=
\mbox{Res}_{z=\lambda_2^-}\frac{-1}{\chi_\kappa^-(z)}\int_{\R}e^{-zs}Q(s)ds=
\frac{-1}{(\chi_\kappa^-)'(\lambda_2^-)}\int_{\R}e^{-\lambda_2s}Q(s)ds
>0.
$$
This contradicts to the assumption $y(t) = O(e^{\lambda_1 t}), \ t \to +\infty,$ and shows that $a>0$. 
\end{remark} 
\begin{remark} \label{e1r} The proof of Lemma \ref{tdo} shows that, for  each $(h,c)\in \Int\,\mathcal{D}_{\frak{L}}$, it holds that $y(t) = O(e^{\lambda_2 t}),$ $ t \to +\infty$, even if the sub-tangency conditions of the lemma are not assumed.  Similarly, $\phi(t) = O(e^{\lambda  t}), \ t \to -\infty$.  In order to establish the same growth estimates
for the derivatives $y'(t),\phi'(t)$, we can proceed as follows.  For example, let us consider $y'(t)$ at $+\infty$.  After 
integrating (\ref{kppOaux}) on $(t,+\infty)$, we obtain
 $$
 y'(t)=cy(t)+\int_t^{+\infty}(\alpha_\kappa+N(s))y(s)+ (\beta_\kappa+M(s))y(s-ch)ds =  O(e^{\lambda_2 t}),\  t \to +\infty.
 $$
\end{remark}
\section*{Acknowledgments}
The authors thank Teresa Faria and  Anatoli Ivanov  for useful discussions:  especially we  would like to acknowledge the support of CONICYT (Chile), project MEC 80110006 which allowed the stay  of Dr.  Ivanov in the University of Talca.  Research was also partially supported by CONICYT through PBCT program ACT-56
and by FONDECYT (Chile), project 1110309. A. Gomez was supported by CONICYT programs "Pasant\'ias doctorales
en el extranjero" and "Becas para estudios de doctorado en Chile".

\affiliationone{
Adrian Gomez and  Sergei Trofimchuk\\
   Instituto de Matem\'atica y Fisica\\
   Universidad de Talca, Casilla
   747, Talca\\
   Chile
   \email{adgomez@utalca.cl\\
   trofimch@inst-mat.utalca.cl}}
\end{document}